\documentclass[reqno]{amsart}
\usepackage[utf8]{inputenc}
\usepackage{xcolor,fullpage,hyperref,cleveref,enumitem,enumitem}
\usepackage{tikz,amsmath,amssymb,color,mathtools}
\usepackage{graphicx}
\usepackage[margin=1in]{geometry}
\usepackage{amsthm}
\usetikzlibrary{calc}
\usepackage{float}
\usepackage{pgf}
\usepackage{amsrefs}
\usepackage[mathscr]{euscript}
\usepackage{colonequals}
\usepackage{stackengine}
\usepackage{xcolor}
\usepackage[normalem]{ulem}

\newcommand{\X}{\texttt{X}}
\newcommand{\Pref}{\mathrm{Pref}}
\newcommand{\Sym}{\mathfrak{S}}
\def\O{\mathcal{O}}
\def\inv{^{-1}}
\def\ci{\subseteq}

\newcommand{\lucky}{\mathsf{lucky}}  
\newcommand{\LPF}{\mathsf{LuckyPF}}  
\newcommand{\PF}{\mathrm{PF}}

\newcommand{\Lucky}{\mathsf{Lucky}}

\newcommand{\al}[1]{\begin{align*} #1 \end{align*}}

\theoremstyle{definition} 
\newtheorem{theorem}{Theorem}[section]
\newtheorem{corollary}{Corollary}[theorem]

\newtheorem{example}[theorem]{Example}

\newtheorem{lemma}[theorem]{Lemma}

\newtheorem{definition}[theorem]{Definition}
\theoremstyle{remark}
\newtheorem*{remark}{Remark}

\newcommand{\NN}{\mathbb{N}}

\newcommand{\out}{\mathcal{O}}

\newcommand{\bsy}{\boldsymbol}

\usepackage{thmtools} 
\usepackage{thm-restate}

\title{Enumerating Vector Parking Functions and their Outcomes Based on Specified Lucky Cars}

\author[Ferreri]{Melanie Ferreri}
\address[M.~Ferreri]{Department of Mathematics, William \& Mary, Williamsburg, VA 23187}
\email{\textcolor{blue}{\href{mailto:email}{mjferreri@wm.edu}}}

\author[Harris]{Pamela E. Harris}
\address[P.~E.~Harris]{Department of Mathematical Sciences, University of Wisconsin-Milwaukee, Milwaukee, WI 53211}
\email{\textcolor{blue}{\href{mailto:peharris@uwm.edu}{peharris@uwm.edu}}}

\author[Martinez]{Lucy Martinez}
\address[L.~Martinez]{Department of Mathematics, Rutgers University, Piscataway, NJ 08854}
\email{\textcolor{blue}{\href{mailto:lucy.martinez@rutgers.edu}{lucy.martinez@rutgers.edu}}}

\author[Swartz]{Eric Swartz}
\address[E.~Swartz]{Department of Mathematics, William \& Mary, Williamsburg, VA 23187}
\email{\textcolor{blue}{\href{mailto:easwartz@wm.edu}{easwartz@wm.edu}}}

\begin{document}

\begin{abstract}

In a parking function, a car is considered \emph{lucky} if it is able to park in its preferred spot. Extending work of Harris and Martinez, we enumerate outcomes of parking functions with a fixed set of lucky cars. We then consider a generalization of parking functions known as \emph{vector parking functions} or \emph{$\boldsymbol{u}$-parking functions}, in which a nonnegative integer capacity is given to each parking spot in the street. 
With certain restrictions on $\boldsymbol{u}$, we enumerate outcomes of $\boldsymbol{u}$-parking functions with a fixed set of lucky cars or with a fixed number of lucky cars. We also count outcomes according to which spots contain lucky cars, and give formulas for enumerating $\boldsymbol{u}$-parking functions themselves according to their set of lucky cars.

\end{abstract}

\maketitle

\section{Introduction}

Fix $n\in \NN=\{1,2,3,\ldots\}$ and let $[n]\coloneq\{1,2,\ldots,n\}$. 
A \emph{parking function} of length $n$ is a sequence $a=(a_1,a_2,\ldots, a_n)\in [n]^n$ whose weakly increasing rearrangement $a^\uparrow=(a_{(1)},a_{(2)},\ldots,a_{(n)})$ satisfies $a_{(i)}\leq i$ for all $i\in[n]$. Parking functions can be described by means of a deterministic parking process as follows. Consider a parking lot with $n$ parking spots on a one-way street (directed from left to right) labeled by $[n]$. A sequence of $n$ cars enters the street from the left one by one, with car $i$ having a preferred spot $a_i$, which we call its \emph{parking preference}. 
For each $i\in \NN$, car $i$ drives to its preferred spot $a_i$ and attempts to park. 
If the spot is not available, the car continues to drive to the right and parks in the next available spot, if one exists. 
If there is no available spot, the car exits the lot and is unable to park. 
We call this the \emph{parking rule}, and the set of preference lists $\alpha=(a_1,a_2,\ldots,a_n)$ is called a \emph{parking function} if all cars can park under the parking rule. 
We let $\PF_n$ denote the set of parking functions of length $n$, and Konheim and Weiss established that $|\PF_n|=(n+1)^{n-1}$ \cite{konheim1966occupancy}. 
For a parking function $\alpha$, the order in which the cars park on the street is called the \emph{outcome} of $\alpha$ and is denoted $\out(\alpha)$. If $\Sym_n$ denotes the set of permutations of $[n]$ written in one-line notation, then the outcome of $\alpha$ is 
\[\out(\alpha)=\pi_1\pi_2\cdots\pi_n,\]
where $\pi_i$ denotes that car $\pi_i$ parked in spot $i$ on the street.

Much has been done in studying parking functions, subsets of parking functions, and generalizations. 
This includes MVP parking functions in which a car bumps an earlier car out of their preference \cite{MVP}, Naples parking functions in which a car backs up attempting to park in the spot behind their preference whenever they find their preferred spot occupied \cites{naples,naplesCatalan,naplesandstats}, cases in which a car only tolerates parking in a subset of spots on the street (with some specified order in which they check those spots) \cites{SamSpiro, pullback, ellinterval, Colaric2020IntervalPF,fang2024vacillatingparkingfunctions,bradt2024unitintervalparkingfunctions}, and cases in which some spots are unavailable at the start of parking \cite{completions}. 
There has also been work in parking cars of various lengths \cites{assortments,countingassortments}, and others in which cars are only allowed to stay in their parking spot a fixed length of time \cite{metered}. 
Moreover, certain families of parking objects have connections to noncrossing partitions \cite{stanley1997parking}, computing volumes of flow polytopes \cite{flowpolys}, counting Boolean intervals in the weak Bruhat order of the symmetric group \cite{boolean}, ideal states in the Tower of Hanoi game \cite{hanoi}, the faces of the permutohedron \cite{unit_perm}, and the Quicksort algorithm \cite{quicksort}.
For more on parking functions and other connections, see the survey by Yan \cite{Yan2015}, and for open problems in this area see \cite{PFCYOA}.

As with permutations, there is a desire to understand parking objects based on some statistic. 
Such work includes the discrete statistics of descents, ascents, ties, peaks, valleys, see 
\cites{statsinPFs,celano2025statisticsellintervalparkingfunctions,DescentsInPF}.
One natural statistic of parking functions is the number of cars that park in their preference; such cars are called \emph{lucky}. 
Given a parking function $\alpha\in\PF_n$, we let $\Lucky(\alpha)=\{i\in[n]: \mbox{car } i \mbox{ is lucky} \}$ and $\lucky(\alpha)=|\Lucky(\alpha)|$.

In the case of parking functions, Gessel and Seo \cite{GesselSeo} show that 
\begin{equation}
\label{eq:gesselseo}
\sum_{\alpha\in\PF_n}q^{\lucky(\alpha)}=q\prod_{i=1}^{n-1}(i+(n-i+1)q). 
\end{equation}
A bijective proof of this result was given by Shin, which also demonstrates that parking functions with a given lucky set are in bijection with labeled forests with the same set of \emph{leaders}, i.e., vertices with the property that they are minimal among their descendants \cite{shin2008newbijectionforestsparking}. However, the formula in \Cref{eq:gesselseo} enumerates parking functions with a fixed number of lucky cars and does not account for which cars are lucky.

In this paper, we study lucky sets as ``admissible sets'' with regard to the $\lucky$ statistic for generalized parking functions.
For a given permutation statistic, a set $S$ is said to be \emph{admissible} if there exists a permutation whose instances of that statistic appear exactly at the indices given by $S$. Admissible sets for permutations were introduced by Billey, Burdzy, and Sagan in studying the peak statistic of permutations  \cite{BilleyBurdzySagan}. Similarly, admissible sets have been studied
for various other permutations statistics. Davis, Nelson, Petersen and Tenner considered admissible pinnacle sets \cite{PinnacleSet} and  Gonz\'alez et al. studied the analogous question for signed permutations \cite{SignedPinnacles}. Other work includes studying the descent polynomial, which counts permutations according to their descent sets \cite{diazlopez2017descentpolynomials}. On the other hand, admissible sets for statistics on parking functions have rarely appeared in the literature. The first few instances include the work on descent sets for parking functions studied by Cruz et al. \cite{statsinPFs}, and the $\lucky$ statistic on Stirling permutations considered as parking functions by Colmenarejo et al. \cite{colmenarejo2024luckydisplacementstatisticsstirling}. Such examples demonstrate the potential for further connections between the $\lucky$ statistic in parking functions and other combinatorial objects. Motivated by this, Harris and Martinez determined formulas for the number of parking functions with a fixed set of lucky cars \cite{harrismartinez}. In that work, they observed the close connection between a descent in the outcome permutation and lucky cars. For example, the parking function $(1,3,4,1)$ has parking outcome $1423$. 
Harris and Martinez show that whenever $\pi_{i-1}>\pi_i$ in the outcome permutation, then car $\pi_i$ is lucky, because as car $\pi_i$ parks, the spot to its left was occupied by a car later in the queue, which implies that spot was available upon car $\pi_i$ attempting to park. Thus car $\pi_i$ parked in their preference. 
This observation implies that if the set of lucky cars is fixed, then certain parking outcomes are not possible. Moreover, car 1 is always a lucky car as it is the first to park on the street. Our work begins by answering an open problem of Harris and Martinez stated in \cite{harrismartinez} by giving a formula for the number of parking outcomes given a fixed set of lucky cars.

\setcounter{section}{2} 
\setcounter{theorem}{1}
\begin{theorem}
Let $I = \{1, i_1, i_2, \ldots, i_{k-1}\}$ be a lucky set of $\PF_n$ with $1<i_1 <\cdots < i_{k-1} \leq n$. If $\out_n(I)=\{\out(\alpha):\alpha\in\PF_n\mbox{ and }\Lucky(\alpha)=I\}$, then
\[
|\out_n(I)|=k! \prod_{j\in [n]\setminus I} |I \cap [j]| = 2^{i_2-i_1} 3^{i_3-i_2} \cdots (k-1)^{i_{k-1} - i_{k-2}} k^{n -(i_{k-1}-1)}.
\]

\end{theorem}

Our main contributions consider the lucky sets of cars for a generalization of parking functions called vector parking functions, where each parking spot has a nonnegative integer capacity \cites{KungYan, PitmanStanley}. 
To begin, we define vector parking functions and then give a list of our main results. 
Throughout we label the parking spots by $\mathbb{N}$, and let $\bsy{u}=(u_1,\ldots,u_{n})$, where $1\leq u_1\leq \cdots\leq u_{n}$. 
For each $1\leq i\leq n$, the capacity of spot $i$ is the multiplicity of $i$ in the vector $\bsy{u}$, which we denote by $m_{i}(\bsy{u})$. 
In particular, $\sum_{i\geq 1} m_{i}(\bsy{u})=n$. 
With this notation at hand, a \emph{$\bsy{u}$-parking function} can be described via the following parking process \cite{2024primevectorparking}. A queue of $n$ cars with parking preferences given by $\bsy{a}=(a_1,\ldots,a_{n})\in\NN^n$ enters the lot, and each car drives to its preferred spot and attempts to park. 
A car may park at a spot $j$ if there are fewer than $m_{j}(\bsy{u})$ cars already parked there. 
Otherwise, it attempts to park in the next spot $j+1$, and so on. 
If there are no available spots past spot $j$, the car fails to park. 
We say that $\bsy{a}$ is a \emph{$\bsy{u}$-(vector) parking function} if all cars are able to park under this parking rule. We generally denote vector parking functions with bold symbols.
For example, let $n=4$, $\bsy{u}=(2,2,3,3)$ and $\bsy{a}=(1,3,3,1)$. Then we have $m_2(\bsy{u})=2, m_3(\bsy{u})=2$, and $m_{1}(\bsy{u})=0$. Thus, there are $2$ non-empty spots with total capacity $4$: spots~$2$, and~$3$ with capacities~$2$, and $2$, respectively. 
The cars with parking preferences given by $\bsy{a}$ then park as follows: Car $1$ parks in spot~$2$, car $2$ parks in spot $3$, car $3$ parks in spot $3$, and car $4$ parks in spot~$2$. In \Cref{fig:ps example}, we show the street after all of the cars have parked. Note that the spot marked with an $\X$ has capacity zero, i.e.\ this spot is unavailable.

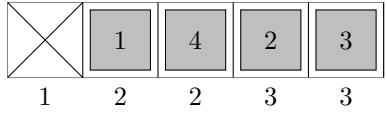
\begin{figure}[!h]
    \centering
    \begin{tikzpicture}
    \draw[step=1cm,gray,very thin] (0,0) rectangle (5,1);
    \foreach \i in {1,2,3,4,5}
    {\draw (\i,0) -- (\i,1);
    }
    \foreach \i in {0}{
    \draw (\i,0)--(\i+1,1);
    \draw (\i+1,0)--(\i,1);
    }
    \draw[fill=gray!50] (1.1,0.1) rectangle (1.9,.9);
    \node at (1.5,.5) {$1$};
    \draw[fill=gray!50] (2.1,0.1) rectangle (2.9,.9);
    \node at (2.5,.5) {$4$};
    \draw[fill=gray!50] (3.1,0.1) rectangle (3.9,.9);
    \node at (3.5,.5) {$2$};
    \draw[fill=gray!50] (4.1,0.1) rectangle (4.9,.9);
    \node at (4.5,.5) {$3$};
    
    \node at (0.5,-.25) {$1$};
    \node at (1.5,-.25) {$2$};
    \node at (2.5,-.25) {$2$};
    \node at (3.5,-.25) {$3$};
    \node at (4.5,-.25) {$3$};
    \end{tikzpicture}
    \caption{Vector parking process with $\bsy{u}=(2,2,3,3)$ and $\bsy{a}=(1,3,3,1)$.}
    \label{fig:ps example}
\end{figure}

When $\bsy{u} = (u_1, u_2, \ldots, u_n)$ is a weakly increasing sequence of positive integers, we let $\PF(\bsy{u})$ denote the set of $\bsy{u}$-parking functions of length $n$. 
A characterization of $\bsy{u}$-parking functions is as follows. 
A preference list sequence $\bsy{a} = (a_1, a_2, \ldots, a_n) \in \mathbb{N}^n$, is in $\PF(\bsy{u})$ if the rearrangement into weakly increasing order $\bsy{a}^\uparrow=(a_{(1)},a_{(2)},\ldots,a_{(n)})$ satisfies $a_{(i)} \leq u_i$ for each $i\in[n]$. 
In this way, when $\bsy{u}=(1,2,\ldots,n)$, the set of $\bsy{u}$ parking functions is precisely the set of classical parking functions. Kung and Yan \cite{KungYan} show that when $\bsy{u}$ is an arithmetic progression, that is $u_i=a+b(i-1)$ for some $a,b \in \NN$, the number of $\bsy{u}$-parking functions is $a(a+bn)^{n-1}$. 
For general $\bsy{u}$, the number of $\bsy{u}$-parking functions is given by a determinantal formula \cites{KungYan, PitmanStanley}, which is not easy to evaluate.

In the classical parking function setting, the parking outcome can be described by a permutation of $[n]$ in one-line notation $\pi=\pi_1\pi_2\cdots\pi_n$, where $\pi_i=j$ means that car $j$ parked in spot $i$. 
We introduce the outcome of a $\bsy{u}$-parking function, which we formally state in \Cref{def:outcome of vector parking}. 
Given the capacities of the parking spaces in a $\bsy{u}$-parking function, 
whenever multiple cars park in the same numbered spot, we adopt the convention of listing those cars in increasing order.
Moreover, whenever there are unavailable parking spots in a $\bsy{u}$-parking function, we denote those with an $\X$.
Namely, if $\bsy{u}=(u_1,u_2,\ldots,u_n)$, with $M=\max(\bsy{u})$, then
the \emph{outcome} of $\bsy{a}=(a_1,a_2,\ldots,a_n)$ is denoted by \[ \out_{\bsy{u}}(\bsy{a})=B_1B_2\ldots B_M,\] where $B_i$ contains the set of cars (in increasing order) that parked in spot $i$, and if a spot $i$ is unavailable (i.e., because $m_i(\bsy{u})=0$), then $B_i=\X$, which means no cars can park there.
In this way, {when ignoring the unavailable spots, i.e.\ any $B_i=\X$}, the outcome of a $\bsy{u}$-parking function is an {ordered} set partition of $[n]$ where the size of the blocks are encoded by positive capacities for the spots described by $\bsy{u}$. For example, the outcome of the $(2,2,3,3)$-parking function  $(1,3,3,1)$, illustrated in \Cref{fig:ps example} is $\out((1,3,3,1))=
\X\{1,4\}\{2,3\}$.

In a vector parking function, a car is said to be \emph{lucky} if it parks in its preferred parking spot. For example, in \Cref{fig:ps example}, the only lucky cars are $2$ and $3$ as they park in their preferred spot $3$, while cars $1$ and $4$ park in spot $2$ making them unlucky. 
We study vector parking functions with a fixed set of lucky cars. We let $\Lucky_{\bsy{u}}(\bsy{a})$ denote the set of lucky cars of $\bsy{a} \in \PF(\bsy{u})$. Unlike the classical parking functions in which the first car is always lucky, this is not always the case for the vector parking functions, as some parking spots might be unavailable at the start of the parking process. For example, if $\bsy{u}=(2,2,3,3)$ and $\bsy{a}=(1,3,3,1)$ then $\Lucky_{\bsy{u}}(\bsy{a})=\{2,3\}$. For any subset $I\subseteq[n]$, we define
\[\LPF_{\bsy{u}}(I)\coloneqq\{\bsy{a}\in \PF(\bsy{u})~:~ \Lucky_{\bsy{u}}(\bsy{a})=I\},\]
which is the set of $\bsy{u}$-parking functions of length $n$ with lucky cars in the set $I$ and unlucky cars in the set $[n]\setminus I$.

Similar to classical parking functions, not every ordered set partition can be the outcome of a $\bsy{u}$-parking function with a fixed set of lucky cars. For example, if $\bsy{u}=(2,2,3,3)$, then the outcome $\out((1,3,3,1))=\X\{1,4\}\{2,3\}$ given in \Cref{fig:ps example} cannot be the outcome of a $\bsy{u}$-parking function with lucky set $I=\{3\}$. This is because car $2$ and car $3$ parked in spot $3$, however when car $2$ enters the street, spot $2$ is not fully occupied. 
This means that parking in spot $3$ would make car $2$ lucky. 

We let $\out_{\bsy{u}}(I)$ denote the set of outcomes of all $\bsy{u}$-parking functions with lucky set $I$. 
In \Cref{subsec: k cars lucky no rep} we  characterize the possible outcomes for $\bsy u$-parking functions given a lucky set $I$, and enumerate outcomes of vector parking functions with no repetitions and where $I$ has a fixed size. 

\setcounter{section}{3} 
\setcounter{theorem}{6}
\begin{theorem}
Fix a lucky set of cars $I$, let $M=\max{(\bsy{u})}$, and let $B=B_1B_2\ldots B_M\in \O_{\bsy{u}}
\coloneqq\{\out_{\bsy{u}}(\bsy{a})~:~\bsy{a}\in\PF(\bsy{u})\}$.
Then $B \in \O_{\bsy{u}}(I)$ 
if and only if the following hold:
\begin{enumerate}
    \item if $B_1\neq \X$, then $B_1\subseteq I$,
    \item if there exists a $b_1\in B_{i-1}\neq \X$ and $b_2\in B_{i}\neq \X$ such that $b_1>b_2$ then $b_2\in I$.
\end{enumerate}
\end{theorem}

In Sections \ref{subsec: spot 1 is unavailable}-\ref{subsec: spot 3 is unavailable}, we enumerate outcomes for a general $I$ in certain special cases for $\bsy u$. In \Cref{sec:outcomesbyspots} we enumerate the outcomes from \Cref{thm:upfs characterization} with a fixed set of \emph{lucky spots}, defined in that section. Lastly in \Cref{sec:uparkingfunctions}, we give a formula for the number of $\bsy u$-parking functions with a fixed set of lucky cars or with $k$ lucky cars for a general vector $\bsy u$, and discuss related results. We conclude with \Cref{sec:future work} by providing some directions for future work.

\setcounter{section}{1}
\section{Counting outcomes of classical parking functions with a fixed set of lucky cars}

In this section, we provide a formula for the number of outcomes of parking functions with a fixed lucky set. Harris and Martinez characterized the permutations that are outcomes of parking functions with a fixed lucky set $I$. We recall this result next.

\begin{theorem}[\cite{harrismartinez}, Theorem 2.3]\label{thm:characterizing outcomes}
Fix a lucky set $I\subseteq[n]$ of $\PF_n$. 
Then $\pi=\pi_1\pi_2\cdots\pi_n\in \Sym_n$ is the outcome $\alpha\in\LPF_n(I)$
if and only if 
\begin{enumerate}
    \item $\pi_1\in I$,
    \item if $\pi_i\in I$ and $\pi_{i+1}\notin I$, then $\pi_i<\pi_{i+1}$, and 
    \item if $\pi_{i-1}>\pi_{i}$, then $\pi_{i}\in I$.
\end{enumerate}
\end{theorem}

We now provide a formula that counts the permutations described in \Cref{thm:characterizing outcomes}.

\begin{theorem}\label{thm:count for outcomes}
Let $I = \{1, i_1, i_2, \ldots, i_{k-1}\}$ be a lucky set of $\PF_n$ with $1<i_1 <\cdots < i_{k-1} \leq n$. If $\out_n(I)=\{\out(\alpha):\alpha\in\PF_n\mbox{ and }\Lucky(\alpha)=I\}$, then
\[
|\out_n(I)|=k! \prod_{j\in [n]\setminus I} |I\cap [j]| = 2^{i_2-i_1} 3^{i_3-i_2} \cdots (k-1)^{i_{k-1} - i_{k-2}} k^{n -(i_{k-1}-1)}.
\]
\end{theorem}

\begin{proof}
We first show that any permutation satisfying condition (3) in \Cref{thm:characterizing outcomes} also satisfies condition (2). Let $\pi \in \Sym_n$. Suppose for all $i \in [n] $ such that $\pi_{i-1}> \pi_i$, $\pi_i \in I$. 
Then by reindexing, we have that for all $i \in[n]$, if $\pi_{i}>\pi_{i+1}$, then $\pi_{i+1}\in I$. By the contrapositive, we have that
if $\pi_{i+1} \not\in I$, then $\pi_{i} < \pi_{i+1}$, which implies that (2) holds for $\pi$ as well.

We proceed to count permutations in $\Sym_n$ satisfying conditions (1) and (3). 
Given a lucky set $I$, let $k = |I|$, and let $J = [n] \setminus I$. 
Let $\pi$ be a permutation satisfying the desired conditions. 
Suppose \[\pi\inv(I) = \{a_1=1, a_2, \ldots, a_k\},\] 
where $a_1 < a_2 < \cdots < a_k$. 
Then $\pi_1 \in I$, and $\pi$ can be partitioned into $|I| = k$ increasing subsequences, where each subsequence begins with some $\pi_{a_i}^* \in I$ and the remaining entries in the subsequence are not in $I$. More technically, $\pi$ (written in one-line notation) is of the form
\[
\pi = \underline{\pi_1^* \pi_2 \cdots \pi_{a_2-1}}\,\, \,\underline{\pi_{a_2}^* \pi_{a_2+1} \cdots \pi_{a_3-1}} \,\, \,\underline{\pi_{a_3}^* \pi_{a_3+1} \cdots \pi_{a_4-1}}  \,\, \,\cdots \,\,\,\underline{\pi_{a_k}^* \pi_{a_k+1} \cdots \pi_{n}},
\]
where the starred entries are the elements of $I$, and the non-starred entries are the elements of $[n] \setminus I$. 
By condition (3), each underlined subsequence must be increasing. 
We now count the number of possible permutations $\pi$ of this form. 

Let $j \in J$. In the permutation $\pi$, $j$ must appear in one of the $k$ subsequences after some $\pi_{a_i}^*$. Since each subsequence is increasing, $j$ must appear in one of the subsequences beginning with a $\pi_{a_i}^*$ that is less than $j$.
Define the set $I_j = \{x\in I: x<j\} = I \cap [j]$. It follows that $|I_j|$ is the number of possible sequences in which $j$ can be placed. Once a subsequence is chosen, the position of $j$ within that subsequence is fixed, since each subsequence is increasing. Then the number of possible placements of all $j\in J$ into the $k$ subsequences is given by $\prod_{j\in J} |I_j|$.
Once all of the $k$ underlined subsequences are constructed, they can be rearranged in $k!$ ways. So the total number of permutations of the desired form is $k! \prod_{j\in [n]\setminus I} |I_j|$.

We check that this product is equivalent to $2^{i_2-i_1} 3^{i_3-i_2} \cdots (k-1)^{i_{k-1} - i_{k-2}} k^{n -(i_{k-1}-1)}$. Let $j \in J$.
Then, $|I_j|=1$ if and only if $1 <j <i-1$, and the
number of $j\in J$ such that $|I_j| = 1$ is given by the number of integers between $1$ and $i_1$; that is, $i_1-2$.
Now let $\ell \in \{2,\ldots, k-1\}$. Then $|I_j| = \ell$ if and only if $i_{\ell-1}<j <i_\ell$. 
It follows that the number of $j\in J$ such that $|I_j| = \ell$ is given by the number of integers  between $i_{\ell-1}$ and $i_\ell$, which is $i_{\ell}-i_{\ell-1}-1$.
Also, the number of $j\in J$ such that $|I_j| = k$ is given by the number of integers in $[n]$ greater than $i_{k-1}$, which is $n-i_{k-1}$.
Then
\al{
\prod_{j\in [n]\setminus I} |I_j| &= (1)^{|\{ j \in [n
] \setminus I: |I_j| = 1\}|}(2)^{| \{ j \in [n] \setminus I: |I_j| = 2\}|} \cdots (k-1)^{| \{ j \in [n] \setminus I: |I_j| = k-1\}|}(k)^{| \{ j \in [n] \setminus I: |I_j| = k\}|}\\
&= (1)^{i_1-2}(2)^{i_2-i_1-1} \cdots(k-1)^{i_{k-1}-i_{k-2}-1} k^{n-i_{k-1}}.
}
By substitution, we obtain
\al{
k! \prod_{j\in [n]\setminus I} |I_j| &= k! (1)^{i_1-2}(2)^{i_2-i_1-1} \cdots (k-1)^{i_{k-1}-i_{k-2}-1}k^{n-i_{k-1}}\\
&= (k)(k-1)\cdots(2)(1) (1)^{i_1-2}(2)^{i_2-i_1-1} \cdots (k-1)^{i_{k-1}-i_{k-2}-1} k^{n-i_{k-1}}\\
&= (1)^{i_1-1}(2)^{i_2-i_1} \cdots (k-1)^{i_{k-1}-i_{k-2}} k^{n-i_{k-1}+1}\\
&= 2^{i_2-i_1} 3^{i_3-i_2} \cdots (k-1)^{i_{k-1} - i_{k-2}} k^{n -(i_{k-1}-1)}.\qedhere
}
\end{proof}

\begin{example}
 We provide an example using the formula in \Cref{thm:count for outcomes} by using the same example in \cite[Example 2]{harrismartinez}, which counts the number of outcome permutations for $I = \{1,4\}$ when $n=5$ to be four. 
 If $I = \{1,4\}$, then $J=\{2,3,5\}$. 
Then $I_2 = \{1\}$, $I_3 = \{1\}$, $I_5=\{1,4\}$.
So the number of outcomes we obtain is 
\[
|I|! \cdot |I_2| \cdot |I_3| \cdot |I_5| = 2!(1)(1)(2) = 4.
\]
\end{example}

We recall the following result with a formula that counts the number of outcomes when the first $k$ cars are lucky.

\begin{theorem}{\cite[Theorem 2.4]{harrismartinez}}
 Let $I=\{1,2,3,\ldots,k\}\subseteq[n]$ be a lucky set of $\PF_n$. Then 
\[|\out_n(I)|=\sum_{J=\{j_1=1,j_2,j_3,\ldots,j_k\}\subseteq[n]}k!\binom{n-k}{j_2-j_1-1,j_3-j_2-1,\ldots,j_{k}-j_{k-1}-1,n-j_{k}}.\] 
\end{theorem}

Using our formula in \Cref{thm:count for outcomes}, we can obtain the same result, and we remark that this can also be recovered by applying the multinomial theorem to the sum.

\begin{corollary}\label{cor:formula}
If $I = \{1,2,\ldots, k\}$ is a lucky set of $\PF_n$ with $k\leq n$, then
\[ |\out_n(I)|=k! \prod_{j\in \{k+1,\ldots, n\}} |I| \,\,=\,\, k!(k)^{n-k}.\]
\end{corollary}

We provide an example that uses the formula in \Cref{cor:formula} and the formula in \cite{harrismartinez} next.

\begin{example} 
If $I = \{1,2,3\}$ and $n=5$, \cite[Theorem 2.4]{harrismartinez} yields
\[|\out_n(\{1,2,3\})|=\sum_{J=\{j_1=1,j_2,j_3,\ldots,j_k\}\subseteq[n]}k!\binom{n-k}{j_2-j_1-1,j_3-j_2-1,\ldots,j_{k}-j_{k-1}-1,n-j_{k}},\]
where in this case $k=3$ and the possible subsets $J$ are $\{1,2,3\}, \,\{1,2,4\}, \{1,2,5\}, \,\{1,3,4\}, \{1,3,5\}, \,\{1,4,5\}$.
 
The sum can be written as
\al{
|\out_n(\{1,2,3\})|&= 3! \binom{2}{0,0,2} + 3! \binom{2}{0,1,1} + 3! \binom{2}{0,2,0}+ 3! \binom{2}{1,0,1} + 3! \binom{2}{1,1,0} + 3! \binom{2}{2,0,0}\\
&= 3!(1 + 2+ 1 + 2+2+1) \\
&= 3!(9) = 3!(3)^{5-3}.
}
\end{example}

In \cite{harrismartinez}, the authors also considered the case for when there are more spots than cars and gave a characterization of the outcomes of parking functions with a fixed lucky set $I$. We let $\PF_{m,n}$ denote the set of $(m,n)$-parking functions with $m$ cars and $n$ spots, and we assume that $1\leq m \leq n$. We begin with a definition of the outcomes in a parking function with more spots than cars. 

\begin{definition}[\cite{harrismartinez}, Definition 3.1]\label{def:outcome for mn}
Let $\Sym_{m,n}$ denote the set of permutations of the multiset $\{\X,\ldots,\X\}\cup[m]$ with $\X$ having multiplicity $n-m$. Given a parking function $\alpha=(a_1,a_2,\ldots,a_m) \in \PF_{m,n}$, define the \textit{outcome} of $\alpha$ by \[\mathcal{O}(\alpha)=\pi_1\pi_2\cdots\pi_n\in \Sym_{m,n},\] where $\pi_i=j\in[m]$ denotes that car $j$ parked in spot $i$, and $\pi_i=\X$ indicates that spot $i$ remained vacant.
\end{definition}

For sake of simplicity in our arguments, we assume that $\X>i$ for any $i\in \mathbb{N}$. 
Moreover, to every element $\pi=\pi_1\pi_2\cdots\pi_n\in\Sym_{m,n}$ we will prepend $\pi_0=\X$.

We recall the characterization of the outcomes of parking functions with more spots than cars with a fixed lucky set $I$.

\begin{theorem}[\cite{harrismartinez}, Theorem 3.3]\label{thm:characterizing outcomes in mn}
Fix a lucky set $I\subseteq[m]$  of $\PF_{m,n}$. Then $\pi=\pi_1\pi_2\cdots\pi_n\in \Sym_{m,n}$ is the outcome of $\alpha\in\LPF_n(I)$
if and only if $\X\pi=\X\pi_1\pi_2\ldots\pi_n$ satisfies the following:
\begin{enumerate}
    \item if $\pi_{i-1}=\X$ and $\pi_i\in[m]$, then $\pi_i\in I$, 

    \item if $\pi_i\in I$ and $ \pi_{i+1} \notin I\cup\{\X\}$, then $\pi_i<\pi_{i+1}$,

    \item if 
    $\pi_{i-1},\pi_{i}\in[m]$ and 
    $\pi_{i-1}>\pi_{i}$, then $\pi_{i}\in I$. 
\end{enumerate}
\end{theorem}

We now provide a formula that counts such permutations described in \Cref{thm:characterizing outcomes in mn}.

\begin{theorem}\label{thm:count for outcomes in mn}
If $I = \{1, i_1,i_2,\ldots, i_{k-2},i_{k-1}\}\subseteq[m]$ is a lucky set of $\PF_{m,n}$ with $1<i_1<i_2<\cdots<i_{k-1}\leq n$, then
\begin{align*}
    |\out_{m,n}(I)| &=\frac{(k+n-m)!}{(n-m)!} \prod_{j\in [m]\setminus I} |I\cap [j]|\\
    &=2^{i_2-i_1-1}3^{i_3-i_2-1}4^{i_4-i_3-1}\cdots (k-1)^{i_{k-1}-i_{k-2}-1}k^{m-i_{k-1}}\frac{(n-m+k)!}{(n-m)!}.
\end{align*}
\end{theorem}
\begin{proof}

We observe, similarly as in the proof of \Cref{thm:count for outcomes}, that for two consecutive nonempty parking spots $\pi_{i-1}, \pi_i \in [m]$, the last two conditions are equivalent to condition $(3)$.

Conditions $(2)$ and $(3)$ in \Cref{thm:characterizing outcomes in mn} both only concern the case where there are two consecutive entries in $\pi$ from the set $[m]$. (This is directly stated in (3), and in (2), $\pi_i \in I \ci [m]$ and $\pi_{i+1} \notin \{\X\}$ implies $\pi_{i+1}\in [m]$.)
So if either $\pi_{i-1} \notin [m]$ or $\pi_{i} \notin [m]$, there are no restrictions imposed by conditions (2) and (3). Then we seek to count permutations $\pi \in \Sym_{m,n}$ satisfying
\begin{enumerate}
    \item if $\pi_{i-1}=\X$ and $\pi_i\in[m]$, then $\pi_i\in I$, and
 
       \item for $\pi_{i-1}, \pi_i \in [m]$: if $\pi_{i}\not\in I$ then $\pi_{i-1}<\pi_{i}$.
    \end{enumerate}
    
Then, similarly as in the proof of \Cref{thm:characterizing outcomes}, a permutation $\pi \in \Sym_{m,n}$ satisfying these conditions can be partitioned into subsequences. 
In this case, we have length 1 subsequences containing only the entry $\X$, and $k$ increasing subsequences (also possibly of length 1), each beginning with some $\pi_{a_i}^* \in I$, and with the remaining entries being from $[m] \setminus I$. 
Because each element of $[m]$ that appears directly after an $\X$ is from $I$, we know that all of the subsequences of elements from $[m]$ will begin with an element from $I$. The assumption that $\pi_0 = \X$ also guarantees that this holds when $\pi_1 \in [m]$.

More technically,
\[
\pi = \X \cdots \X \, \underline{\pi_{a_1}^* \pi_{a_1+1} \cdots \pi_{a_2-1}}\,\X \cdots \X \,\underline{\pi_{a_2}^* \pi_{a_2+1} \cdots \pi_{a_3-1}} \,\X \cdots \X  \,\underline{\pi_{a_3}^* \pi_{a_3+1} \cdots \pi_{a_4-1}}\cdots \underline{\pi_{a_k}^* \pi_{a_k+1} \cdots \pi_{m}} \,\X \cdots \X ,
\]
where the starred entries are the elements of $I$, and the non-starred entries are the elements of $[m] \setminus I$. By condition (2), each underlined subsequence must be increasing. The $\X$'s need not appear between every pair of increasing subsequences.
We count the number of multiset permutations of this form.
We first check the possible ways to sort the elements of $[m]\setminus I$ into the subsequences.

Similarly as in the proof of \Cref{thm:characterizing outcomes}, each $j \in [m]\setminus I$ can be placed into one of the subsequences beginning with a $\pi_{a_i}^*$ that is less than $j$. Let $I_j = \{x\in I: x<j\} = I \cap [j]$. Then $|I_j|$ is the number of possible subsequences in which an element $j$ can be placed.
Once a subsequence is chosen, the position of $j$ within that subsequence is fixed, since the subsequence is increasing.
Then, over all $j\in [m]\setminus I$, the number of possible ways to place them into the $k$ different subsequences of elements from $[m]$ is $\prod_{j \in [m]\setminus I} |I_j|$.

Once all of the $k$ subsequences of elements from $[m]$ are constructed, they may appear in any order. 
We also must place $n-m$ $\X$'s throughout the permutation, and these can be placed anywhere between the $k$ subsequences. By construction, this will not contradict condition (1), since each subsequence we have constructed begins with an element of $I$.
Then, to count all multiset permutations of the desired form, we count the number of ways to arrange $k$ distinguishable subsequences and $n-m$ indistinguishable subsequences (i.e., the subsequences of the form ``$\X$''), which is given by $\frac{(k+n-m)!}{(n-m)!}$.
It follows that $|\mathcal{O}_{m,n}(I)|$ is counted by
\[\frac{(k+n-m)!}{(n-m)!} \prod_{j\in [m]\setminus I} |I_j|.\]

We check that this product is equivalent to \[2^{i_2-i_1-1}3^{i_3-i_2-1}4^{i_4-i_3-1}\cdots (k-1)^{i_{k-1}-i_{k-2}-1}k^{m-i_{k-1}}\frac{(k+n-m)!}{(n-m)!}.\] 

Let $j \in [m]\setminus I$. By the same reasoning as in the proof of \Cref{thm:characterizing outcomes}, we have 
\[
\prod_{j\in {[m]\setminus I}} |I_j| = (1)^{i_1-2}(2)^{i_2-i_1-1} (3)^{i_3-i_2-1} \cdots(k-1)^{i_{k-1}-i_{k-2}-1} k^{m-i_{k-1}}.
\]
By substitution, it follows that 
\al{
\frac{(k+n-m)!}{(n-m)!} \prod_{j\in [m]\setminus I} |I_j| &= \frac{(k+n-m)!}{(n-m)!} (1)^{i_1-2}(2)^{i_2-i_1-1}(3)^{i_3-i_2-1} \cdots(k-1)^{i_{k-1}-i_{k-2}-1} k^{m-i_{k-1}}\\
&= \frac{(k+n-m)!}{(n-m)!}(2)^{i_2-i_1-1}(3)^{i_3-i_2-1} \cdots(k-1)^{i_{k-1}-i_{k-2}-1} k^{m-i_{k-1}}.\qedhere
}
\end{proof}

Using the formula in \Cref{thm:count for outcomes in mn}, we obtain a formula for the number of outcomes of parking functions when the first $k$ cars are lucky. 
\begin{corollary}\label{cor:again}
If $I = \{1, 2, 3, \ldots, k\}\subseteq[m]$ is a lucky set of $\PF_{m,n}$, then
\[
|\out_{m,n}(I)| = k^{m-k}\frac{(n-m+k)!}{(n-m)!}.
\]
\end{corollary}
We note that \Cref{cor:again} gives an alternative expression for the count given in \cite[Theorem 3.4]{harrismartinez}.
Next we determine how many outcomes there are such that exactly $k$ cars are lucky. Note that this cannot be obtained by simply summing the formula above over all possible lucky sets of size $k$, since this would cause us to overcount. See \Cref{ex:two lucky set one outcome}.

\begin{example}\label{ex:two lucky set one outcome}
If $m=n=5$, the parking function $(1,2,1,1,1)$ has outcome $\pi=12345$ and lucky set $\{1,2\}$.  Also, the parking function $(1,1,3,1,1)$ has outcome $\pi'=12345$ and lucky set $\{1,3\}$. So this single outcome occurs in correspondence with two different lucky sets of the same size.
\end{example}

\begin{lemma}
If $S$ is an outcome of an $(m,n)$-parking function in which $k$ cars are lucky, then $S$ is also an outcome for a parking function in which $\ell$ cars are lucky, for any $k<\ell\leq m$. Consequently the number of outcomes in which at most $k$ cars are lucky is the same as the number of outcomes in which exactly $k$ cars are lucky.
\end{lemma}
\begin{proof}
Suppose that $k$ cars in $S$ are lucky. We may specify $\ell -k$ unlucky cars in $S$ and alter the corresponding parking function so that those cars' preferences are now the spots where they parked. This yields a parking function with outcome $S$ in which now $\ell$ cars are lucky. So any outcome for which $k$ cars are lucky is also an outcome for which $\ell>k$ cars are lucky. 
    
That is, if $\out_{m,n}(k)$ denotes the set of outcomes of $(m,n)$-parking functions with exactly $k$ lucky cars, then $\out_{m,n}(1) \subseteq \out_{m,n}(2) \subseteq \cdots \subseteq \out_{m,n}(m)$. The number of parking functions in which at most $k$ cars are lucky is given by
\[
\bigcup_{i=1}^k \out_{m,n}(i) = \out_{m,n}(k),
\]
since all sets in the union are equal to or contained in $\out_{m,n}(k)$.
\end{proof}

Let $\genfrac{\langle}{\rangle}{0pt}{}{n}{k}$ denote the Eulerian number $A(n,k)$, i.e., the number of permutations of $[n]$ with exactly $k$ descents. We now connect this sequence to the number of outcomes of $(m,n)$-parking functions.

\begin{theorem}\label{thm:mn pf with k lucky}
    Let $1\leq m \leq n $. The number of outcomes for $(m,n)$-parking functions in which exactly $k$ cars are lucky is given by
    \[   \left|\bigcup_{|I|=k}\out_{m,n}(I) \right|=\sum_{j=0}^{k-1} \left[\sum_{d=0}^{k-j-1} \binom{m-j-1}{d}  \binom{n-m+j+1}{j+1+d}\right] \genfrac{\langle}{\rangle}{0pt}{}{m}{j}.
    \]
\end{theorem}

\begin{proof}
As noted earlier, only lucky cars can form descents with the car to their left. Suppose there are a total of $k$ lucky cars, and in the outcome, we have $j$ lucky cars that form descents; that is, $j$ lucky cars which park directly to the right of a car of higher index. Since the car that parks in the leftmost occupied spot is always lucky, it follows that $j\leq k-1$.

Also, for $n\geq m$, we have $n-m$ unoccupied spots in the outcome. We count the possible outcomes with $j$ descents and $n-m$ unoccupied spots, for all possible $j$.

Fix $j \in \{0,1,\ldots, k-1\}$. Omitting the unoccupied spots, we may view the outcome as a  permutation of $[m]$ with $j$ descents. There are $\genfrac{\langle}{\rangle}{0pt}{}{m}{j}$ such options.

Let $d$ be the number of lucky cars that park directly to the right of an unoccupied spot, not including the leftmost car (which may or may not appear to the right of an unoccupied spot). Then $d+j \leq k-1$, since the leftmost car is always lucky and is not included in those sets counted by $j$ and $d$. 
Among the $m$ total cars, we have already selected $j$ of them to form descents, and the leftmost one is excluded from the count, so we select $d$ among the remaining $m-j-1$ cars to park directly to the right of unoccupied spots. There are $\binom{m-j-1}{d}$ ways to choose them. 

Once these $d$ cars are chosen, we know the positions of $d$ unoccupied spots in the outcome. It remains to place the remaining $n-m-d$ unoccupied spots. These cannot be placed to the left of unlucky cars, since this would cause those cars to be lucky. If they were to be placed directly to the left of a lucky car that did not form a descent, other than the leftmost lucky car, they would have already been placed when the $d$ such cars were selected. It remains that the $n-m-d$ unoccupied spots may be placed either to the left of one of the $d$ existing unoccupied spots, to the left of one of the $j$ descents, at the far left before all occupied spots, or at the far right, after all cars. We may view this as having $n-m-d$ $\X$'s and $j+d+1$ {contiguous subsequences whose restriction to integers form increasing runs}, beginning with either descents, the leftmost car, or the $d$ unoccupied spots placed earlier, and we must choose how to place the $\X$'s among them. These runs appear in a fixed relative order. The number of such arrangements is counted by \[ \binom{n-m-d+j+d+1}{j+d+1} = \binom{n-m+j+1}{j+d+1}.\]

To find the total number of possible outcomes, we apply these counts summing over all possible $j$ and $d$. This yields
\[
    \sum_{j=0}^{k-1} \genfrac{\langle}{\rangle}{0pt}{}{m}{j} \left[\sum_{d=0}^{k-j-1} \binom{m-j }{d} \binom{n-m+j+1}{ j+d+1} \right],
\]  
as claimed.
\end{proof}

Simplifying the formula in \Cref{thm:mn pf with k lucky} to the case where $m=n$, we have
\begin{align}
\label{eq:1}\sum_{j=0}^{k-1}\left[\sum_{d=0}^{k-j-1}\binom{n-j-1}{d}\binom{j+1}{j+1+d}\right]\genfrac{\langle}{\rangle}{0pt}{}{n}{j}.\end{align}
Now notice $\binom{j+1}{j+1+d}=0$ for all $d\geq 1$, so we can further simplify \Cref{eq:1} to

\begin{align*}
\sum_{j=0}^{k-1}\left[\binom{n-j-1}{0}\binom{j+1}{j+1}\right]\genfrac{\langle}{\rangle}{0pt}{}{n}{j}=&\sum_{j=0}^{k-1}\genfrac{\langle}{\rangle}{0pt}{}{n}{j},
\end{align*}
which counts the permutations with at most $k-1$  descents. We recall that every descent in an outcome forces a lucky car. Moreover, the car parking in the first spot is always lucky, but this is not a descent of the permutation. Hence we see that this gives the count of possible outcomes having exactly $k$ lucky cars. Thus, we have established the following.
\begin{corollary}
    For any $n\geq 1$, the number of outcomes of elements in $\PF_n$ having exactly $k$ lucky cars is given by 
    $\sum_{j=0}^{k-1}\genfrac{\langle}{\rangle}{0pt}{}{n}{j}$,
which counts the permutations with at most $k-1$  descents.
\end{corollary}

\section{Counting outcomes for vector parking functions}\label{sec:outcomes for vector pfs}

In this section, we begin by describing all of the possible lucky sets for $\bsy{u}$-parking functions. 
Then, we consider counting outcomes for families of $\bsy{u}$-parking functions for various $\bsy{u}$. We first recall, as described in \cite[pg.\ 4]{2024primevectorparking}, that $\bsy{u}$ is a vector that encodes the capacities of the parking spots on the street. If $\bsy{u}=(u_1,u_2,\ldots,u_n)$ with $1\leq u_1\leq u_2\leq \cdots \leq  u_n $ and $u_i\in \NN=\{1,2,\ldots\}$ then the capacity of spot $i$ is the multiplicity of $i$ in the vector $\bsy{u}$.

\begin{definition}
Let $\bsy{a}=(a_1,a_2,\ldots,a_{n})$ be a $\bsy{u}$-parking function. We say that $\bsy{a}$ has \emph{lucky set} $I$ if all cars indexed by $I$ can park in their preferred spot, and all other cars park in a spot other than their preference. We denote this by $\Lucky_{\bsy{u}}(\bsy{a})=I$.
\end{definition}

The first observation is that in contrast to classical parking functions, $\bsy{u}$-parking functions need not have~$1$ appear in their lucky sets. Consider the following example.

\begin{example}
If $\bsy{u}=(1,1,3,3,3)$ and $\bsy{a}=(2,1,2,1,1)$, then $\Lucky_{\bsy{u}}(\bsy{a})=\{2,4\}$ because cars $2$ and $4$ park in their preferred spot $1$ and cars $1,3$ and $5$ park in spot $3$.
\end{example}

\begin{lemma}\label{lemma: if spot 1 is available}
If $m_{1}(\bsy{u})=m$, i.e. the multiplicity of $1$ in $\bsy{u}$ is $m$, then every car with preference $1$ is lucky.
\end{lemma}
\begin{proof}
As in the classical case, the car that parks in spot $1$ is lucky \cite[Theorem 2.3]{harrismartinez}. Hence, if spot $1$ has capacity $m$, all $m$ cars that park in spot $1$ are lucky.
\end{proof}

Next, we formally define the outcome of a $\bsy{u}$-parking function.

\begin{definition}\label{def:outcome of vector parking}
Let $\bsy{u}=(u_1,u_2,\ldots,u_n)$, with $M=\max(\bsy{u})$.
Given a $\bsy{u}$-parking function $\bsy{a}=(a_1,a_2,\ldots,a_n)$, the \emph{outcome} of $\bsy{a}$ is denoted by \[ \out_{\bsy{u}}(\bsy{a})=B_1B_2\ldots B_M,\] 
where $B_i$ contains the set of cars (in increasing order) that parked in spot $i$, and if a spot $i$ is unavailable (i.e., because $m_i(\bsy{u})=0$), then $B_i=\X$, which means no cars can park there.
We denote by $\out_{\bsy{u}}$ the set of all possible outcomes of $\bsy{u}$-parking functions.
\end{definition}

\begin{example}\label{ex:outcome def and lucky}
If $n=5$, $\bsy{u}=(1,1,3,3,3)$, and $\bsy{a}=(2,1,2,1,2)$, then the outcome is the ordered set partition of $\{\X\}\cup[5]$ given by $B=B_1B_2B_3=\{2,4\}\X\{1,3,5\}$,  because cars $2$ and $4$ parked in spot $1$, cars $1, 3$ and $5$ parked in spot $3$, and spot $2$ is unavailable since it does not appear in the vector $\bsy{u}$. 
Furthermore, cars $2$ and $4$ both park in their preferred spot; namely, spot $1$. 
All other cars park in a spot other than their preferred spot. 
Hence, the $\bsy{u}$-parking function $\bsy{a}$ has lucky set $I = \{2,4\}$.
\end{example}

Next, we observe that not every ordered set partition can be the outcome of a $\bsy{u}$-parking function with a fixed set of lucky cars. 

\begin{example}
Let $\bsy{a}$ be a $\bsy{u}$-parking function with $\bsy{u}=(2,3,5,5)$ and suppose that $\Lucky_{\bsy{u}}(\bsy{a})=\{3\}$. 
Notice that $B=\X\{4\}\{2\}\X\{1,3\}$ is not a possible outcome since when car $2$ entered the street, car $4$ had not entered the street and so spot $2$ would be empty. This means that in order for car $2$ to park in spot $3$, it would have to be its preference. Thus, the only way for car $2$ to park in spot $3$ is if it preferred that spot, in which case car $2$ would be lucky, implying $2\in \Lucky_{\bsy{u}}(\bsy{a})$, which is a contradiction. 
\end{example}

In the case of the classical parking functions, descents forced certain cars to be lucky. This still holds for $\bsy{u}$-parking functions. However, if a car follows an unavailable spot, then that car may or may not be lucky. We now give a characterization of the types of outcomes that arise in a $\bsy{u}$-parking function with a fixed set of lucky cars. Recall that $ \LPF_{\bsy{u}}(I)$ denotes the set of $\bsy{u}$-parking functions that have lucky set of cars $I$ and that $\out_{\bsy{u}}(I)$ denotes the set of outcomes of all $\bsy{u}$-parking functions with lucky set $I$.

\begin{theorem}\label{thm:upfs characterization}
Fix a lucky set of cars $I$, let $M=\max{(\bsy{u})}$, and let $B=B_1B_2\ldots B_M\in \O_{\bsy{u}}\coloneqq\{\out_{\bsy{u}}(\bsy{a})~:~\bsy{a}\in\PF(\bsy{u})\}$.
Then $B \in \O_{\bsy{u}}(I)$
if and only if the following hold:
\begin{enumerate}
    \item if $B_1\neq \X$, then $B_1\subseteq I$,
    \item if there exists a $b_1\in B_{i-1}\neq \X$ and $b_2\in B_{i}\neq \X$ such that $b_1>b_2$, then $b_2\in I$.
\end{enumerate}
\end{theorem}

\begin{proof}
Let $B=B_1B_2\ldots B_M$, where $M=\max{(\bsy{u})}$.

($\Rightarrow$) For the forward direction, we verify that when $B=B_1B_2\ldots B_M$ is the outcome of a $\bsy{u}$-parking function with lucky set of cars $I$, it satisfies each of the conditions. 
\begin{enumerate}
    \item This follows by \Cref{lemma: if spot 1 is available}. 
    \item Suppose that there exists a $b_1\in B_{i-1}\neq \X$ and $b_2\in B_{i}\neq \X$ such that $b_1>b_2$. This implies that car $b_2$ parks on the street before car $b_1$ enters the street to park. Thus, when car $b_2$ entered the street to park, spot $i-1$ was not completely occupied. Assume for contradiction that car $b_2$ has a preference $1\leq a_i\leq i-1$. Then car $b_2$ would park in the first available spot past $a_i$. Among the spots numbered $a_i,a_i+1,\ldots, i-1$, at least spot $i-1$ still has not reached capacity. So car $b_2$ would park in spot $x$ satisfying $a_i\leq x\leq i-1<i$, contradicting the assumption that car $b_2$ parks in spot $i$. Thus, car $b_2$ parks in spot $i$ and must have done so by preferring spot $i$. This makes car $b_2$ a lucky car, and so whenever $b_1\in B_{i-1}\neq \X$ and $b_2\in B_i\neq \X$ we have that $b_2\in I$, as claimed.
\end{enumerate}

($\Leftarrow$) Let $B=B_1B_2\ldots B_M $ satisfy conditions $(1)$ and $(2)$ above. We construct a $\bsy u$-parking function whose lucky set is $I$ and outcome is $B$. Let $\bsy{a} = (a_1,\ldots, a_n)$ be a vector of preferences, in which car $i$ prefers spot $a_i$. For each car $h$, we set $a_h$ to be the spot where car $h$ appears in $B$. It follows that the outcome of $\bsy a$ is $B$, and all cars are lucky in $\bsy a$.
We construct a parking function $\bsy{a}'$ in which only the cars indexed by $I$ are lucky. For each car $i$ in $I$, let $a_i' = a_i$. Otherwise for $j \notin I$, let $a_j' = a_j-1$. Since $j \notin I$, condition (1) implies $a_j \neq 1$, so $a_j'\geq 1$. Hence $\bsy{a}'$ is a vector consisting of positive integers with $a_h'\leq a_h$ for all $h$. We know that $\bsy{a}$ is a valid $\bsy u$-parking function since all cars are able to park. Then by the previous inequality, $\bsy{a}'$ also satisfies the definition of a $\bsy u$-parking function.

Suppose for contradiction that the outcome of $\bsy{a}'$ is different from the outcome of $\bsy{a}$. Let $j$ be the car of lowest index whose parking spot differs in the outcomes of $\bsy a$ and $\bsy a'$. Only unlucky cars' preferences are changed.
If $j$ were a lucky car which ended up parking in a different spot, then the parking spot of an earlier unlucky car must have been changed. By minimality of $j$, we have that $j$ is then unlucky.   Then if car $j$ parks in spot $d$ in the outcome of $\bsy a$, its preference is changed to $d-1$ in $\bsy a'$, and, since the spot that $j$ parks in changes in the outcome of $\bsy a'$, we have that $j$ is now able to park in spot ${d-1}$ in the outcome of $\bsy a'$. 
This is only possible if spot ${d-1}$ is available and not entirely occupied when car $j$ attempts to park, which implies that in the outcome of $\bsy a$, there must have been some car $f>d$ which parks in spot $B_{d-1}$.
However, by property (2), this implies $j\in I$, meaning $j$ was not unlucky, which is a contradiction.
Then $\bsy a'$ is a valid $\bsy u$-parking function with outcome $B$ and lucky set $I$. So if $B$ satisfies (1) and (2), then $B \in \out_{\bsy u}(I)$.
\end{proof}

\subsection{Counting outcomes when $k$ cars are lucky and the vector has no repetition}\label{subsec: k cars lucky no rep}
In this section, we consider removing $m$ spots from a street of length $n+m$ and let $n$ be the number of cars. Moreover, we let $\bsy{u}=(u_1,u_2,\ldots, u_n)$ where none of the entries of $\bsy{u}$ repeat. Later, we consider the case in which there are repeated entries in the vector $\bsy{u}$. It turns out that $\bsy{u}$-parking functions are a generalization of \emph{parking completions}, in which all spots have capacity 1, and various spots are presumed to be unavailable. Suppose there are $m$ of the $n+m$ spots already taken. Then, a parking completion $a=(a_1,a_2,\ldots, a_n)$  for a sequence $s=(s_1,s_2,\ldots,s_m)$ is a sequence in which all $n$ cars can successfully park where the entries of $s$ denote which spots are unavailable in increasing order \cite{completions}. In this context, $\bsy{u}$ is a vector that contains the $n$ spots that are available on the street. More precisely, \[\bsy{u}=(1,2,\ldots, s_1-1, s_1+1, s_1+2, \ldots, s_2-1, s_2+1,\ldots, s_{m-1}-1,s_{m-1}+1,\ldots, s_m-1,s_m+1,\ldots, n+m),\] where $\bsy{u}$ is of length $n$.

\begin{example}
Let $n=3$ and $\bsy{u}=(2,4,6)$. Then, $s=(1,3,5)$ and some possible $\bsy{u}$-parking functions or parking completions are $(1, 1, 1),
 (1, 1, 6),
 (1, 3, 6),
 (1, 5, 2),
 (1, 5, 3),
 (1, 5, 4)$.
\end{example}

Throughout this section, we fix the size of a lucky set $I$ and let $|I|=k$. We stress that in fixing the size of the lucky set, we are not concerned with the elements in the lucky set. We give a formula in terms of the Eulerian numbers that counts the number of distinct outcomes when $|I|=k$. We only count distinct outcomes, as two different lucky sets with the same size could yield the same outcome. See \Cref{ex: same outcome}.

\begin{example}\label{ex: same outcome}
Consider $\bsy{u}=(1,2,3,5), k=2$ and consider two lucky sets $I=\{1,3\}$ and $I'=\{1,4\}$. A possible outcome for both lucky sets is $\out_{\bsy{u}}(I)=\out_{\bsy{u}}(I')=\{1\}\{2\}\{4\}\X\{3\}$ with $\bsy{a}=(1,1,5,1)$ and $\bsy{a'}=(1,1,4,3)$, respectively. 
\end{example}

Following notation from \cite{diazlopez2017descentpolynomials}, let $\delta = (\delta_1, \ldots, \delta_k)$ be a weak composition of $n$. We denote the multinomial coefficient \[\binom{n }{ \delta} \coloneqq \frac{n!}{\delta_1!\, \delta_2 ! \,\cdots\, \delta_k!}.\]

The following result can be viewed as a formula for enumerating outcomes of parking completions with $k$ lucky cars.

\begin{theorem}\label{thm:outcomes with m unavailable spots}
Let $\bsy{u}=(u_1,u_2,\ldots, u_n)$ such that all $u_i$ are distinct and in increasing order. Suppose spots $s_1,s_2,\ldots, s_m$ are unavailable. Let $\delta_1 = s_1-1$, $\delta_2 = s_2-s_1-1$, $\ldots$, $\delta_m = s_m-s_{m-1}-1$, $\delta_{m+1} = n+m - s_m$. The number of distinct outcomes with $k$ lucky cars is
\[
\left|\bigcup_{|I| = k}\out_{\bsy{u}}(I)\right|=\binom{n}{\delta}\sum_{\stackanchor{$(i_1,\ldots, i_{m+1})$ nonnegative,}{$\sum_{p}i_p \leq  \max\{k-\delta_1, k-1\}$}} \left( \prod_{j=1}^{m+1} \genfrac{\langle}{\rangle}{0pt}{}{\delta_j}{i_j}
\right),
\]
where 
$\genfrac{\langle}{\rangle}{0pt}{}{\delta_j}{i_j}$
is the Eulerian number with initial conditions $\genfrac{\langle}{\rangle}{0pt}{}{\delta_j}{0}=1$
for any $\delta_j\geq 0$.
\end{theorem}
\begin{proof}

Suppose that spots $s_1, s_2, \ldots, s_m$ are forbidden, and there are $n$ cars to park in the $n$ available spots. We construct a (weak) composition $\delta$ of $n$ as follows: Let $\delta_1 = s_1-1$, $\delta_2 = s_2-s_1-1$, $\ldots$, $\delta_m = s_m-s_{m-1}-1$, $\delta_{m+1} = n+m - s_m$. The nonzero parts of $\delta$ are the lengths of contiguous runs of available spots. This construction is illustrated in \Cref{fig:constructing composition}.

\begin{figure}[!h]
    \centering
    \begin{tikzpicture}
    \draw[step=1cm,gray,very thin] (0,0) rectangle (2,1);
    \draw[step=1cm,gray,very thin] (3,0) rectangle (7,1);
    \draw[step=1cm,gray,very thin] (8,0) rectangle (12,1);
    \draw[step=1cm,gray,very thin] (13,0) rectangle (14,1);
    \foreach \i in {1,2,3,4,5,6,7,8,9,10,11,12,13,14}
    {\draw (\i,0) -- (\i,1);
    }
    \foreach \i in {4,9,10}{
    \draw (\i,0)--(\i+1,1);
    \draw (\i+1,0)--(\i,1);
    }

    \node at (2.5,.5) {$\ldots$};
    \node at (7.5,.5) {$\ldots$};
    \node at (12.5,.5) {$\ldots$};
    \node at (2,-.6) {$\underbrace{\phantom{xxxxxxxxxxxxxxxxxxx}}_{ \delta_1}$};
    \node at (4.5,-.6) {$s_1$};
    \node at (7,-.6) {$\underbrace{\phantom{xxxxxxxxxxxxxxxxxxx}}_{ \delta_2}$};
    \node at (9.5,-.6) {$s_2$};
    \node at (10.5,-.6) {$s_3$};
    \node at (12.5,-.6) {$\underbrace{\phantom{xxxxxxxxxxxxxx}}_{\delta_4} $};
    \end{tikzpicture}
    \caption{An example where there are $3$ unavailable spots: $s_1, s_2$, and $s_3$. The lengths of the contiguous blocks of available spots determine the parts of the weak composition $\delta$. In this example, since $s_2$ and $s_3$ are consecutive unavailable spots, we have $s_3-s_2-1=\delta_3=0 $.}
    \label{fig:constructing composition}
\end{figure}
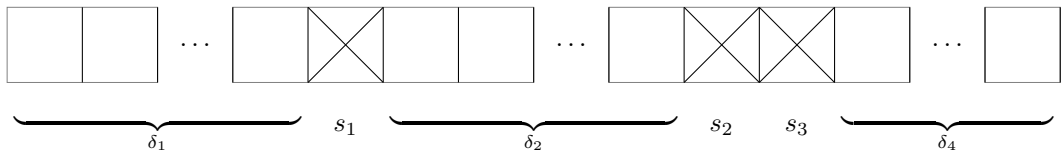

We then distribute the $n$ cars into subsets of sizes $\delta_1, \delta_2, \ldots, \delta_{m+1}$. This determines the block of contiguous available spots in which each car will park. 
There are $\binom{n }{\delta} = \frac{n!}{\delta_1! \cdots \delta_{m+1}!}$ ways to do this.

Once this distribution is chosen, we then may have up to $k-1$ total descents among all blocks (if the first spot is available; i.e.\ $s_1 \neq 1$), or $k$ descents if the first spot is unavailable. Supposing that the $j$th subsequence has $i_j$ descents, we have that $i_1 + \cdots + i_{m+1} \leq \max\{k-\delta_1, k-1\}$.
For each choice of $(i_1,\ldots, i_{m+1})$ we then have $\prod_{j=1}^{m+1} \genfrac{\langle}{\rangle}{0pt}{}{\delta_j}{i_j}$ choices over all subsequences.

In total we obtain
\[
\binom{n }{\delta}\sum_{\stackanchor{$(i_1,\ldots, i_{m+1})$ nonnegative,}{$\sum_{r}i_r \leq  \max\{k-\delta_1, k-1\}$}} \left( \prod_{j=1}^{m+1} \genfrac{\langle}{\rangle}{0pt}{}{\delta_j}{i_j} \right)
\]
possible outcomes when there are $k$ lucky cars and spots $s_1,s_2,\ldots,s_m$ are forbidden.
\end{proof}

\emph{Remark.} For the classical case in which $\bsy{u} = (1,2,\ldots, n)$, we have the composition $\delta = (n)$, and the formula in \Cref{thm:outcomes with m unavailable spots} specializes to \Cref{cor:again}:
\[
\binom{n }{n}\sum_{i \leq k-1} \left( \prod_{j=1}^{1} \genfrac{\langle}{\rangle}{0pt}{}{n}{i} \right) = \sum_{i=1}^{k-1} \genfrac{\langle}{\rangle}{0pt}{}{n}{i}.
\]
So again we have the $(k-1)$th summation of the Euler triangle, counting the number of permutations of $[n]$ with at most $k-1$ descents.

\subsection{Counting outcomes with one unavailable spot}\label{sec:one unavailable spot}

In this subsection, we are interested in enumerating outcomes based on which cars are lucky, regardless of which spots they end up choosing, when there is exactly one unavailable spot.

\subsubsection{Counting outcomes when the only unavailable spot is spot $1$}\label{subsec: spot 1 is unavailable}
In this section, we give a formula for the number of outcomes of $\bsy{u}$-parking functions for $\bsy{u} = (2,3,\ldots, n , n+1)$ with lucky set $I = \{i_1, i_2 \ldots, i_k\}$ where $k\leq n$. In this scenario, there are $n+1$ spots on the street with spot $1$ unavailable and the remaining $n$ spots are available, and for any outcome $B\in \out_{\bsy{u}}(I)$, $B_1=\X$. In this case we can imagine that a lucky car $0$ has first come to occupy the unavailable spot $1$.

\begin{theorem}\label{thm:count for outcomes with forbidden spot one}
Let $\bsy{u} = (2,3,\ldots,n,  n+1)$. If $I = \{i_1, i_2, \ldots, i_{k}\}$ is a lucky set of $\PF_n(\bsy{u})$ with $1\leq i_1 <\cdots < i_{k} \leq n$, then
\begin{align}\label{eq:rhs}
|\out_{\bsy u}(I)|=k! \prod_{j\in [n]\setminus I} (|I \cap [j]|+1) =2^{i_2-i_1}3^{i_3-i_2}\cdots k^{i_{k}-i_{k-1}} (k+1)^{n-i_k}.
\end{align}
\end{theorem}

\begin{proof}
To account for the assumption that spot 1 is the only unavailable spot, we assume that an imaginary car 0 parks in spot 1, and proceed similarly as in the classical case. Each unlucky car $j$ can park after any lucky car $i$ where $i<j$, $i=0$ included. In this case, the number of options for each unlucky car $j$ is the number of entries in $I$ which are less than $j$, plus one to account for spot 0.

Then the $j$th car has $|I \cap [j]| + 1$ options for which increasing subsequence to park in. The sequence starting with 0 is fixed to appear first, and the remaining $k$ other subsequences may appear in any order. It follows that the number of outcomes in this case is 
\[
|\out_n(I)|=k! \prod_{j\in [n]\setminus I}( |I\cap [j]|+1).
\]

Next we establish an alternative count similar to the formula appearing in \Cref{thm:count for outcomes}, which will give the right-hand-side of \eqref{eq:rhs}.

The number of unlucky cars which have only one option is the number of unlucky cars with index less than $i_1$. In this case, these cars must park in the sequence following car 0. There are $i_1-1$ such cars.
The number of unlucky cars which have exactly two options is the number of unlucky cars with index between $i_1$ and $i_2$. There are $i_2-i_1-1$ such cars. Similarly, for $\ell = 3,\ldots, k$, the number of unlucky cars which have exactly $\ell$ options is the number of unlucky cars with index between $i_{\ell-1}$ and $i_\ell$. There are $i_\ell-i_{\ell-1}-1$ such cars.
The number of unlucky cars which have $k+1$ options is the number of cars with index greater than {that of} every car indexed by $I$. There are $n-i_k$ such cars.

Then, to place the unlucky cars into the $k+1$ subsequences, the total number of options is
\[
2^{i_2-i_1-1}3^{i_3-i_2-1}\cdots k^{i_{k}-i_{k-1}-1} (k+1)^{n-i_k}.
\]
We may then arrange the subsequences in any order, as long as the subsequence beginning with car 0 is always first. There are $k!$ such arrangements.
In total the number of outcomes is
\begin{align*}
    k!\,2^{i_2-i_1-1}3^{i_3-i_2-1}\cdots k^{i_{k}-i_{k-1}-1} (k+1)^{n-i_k} = 2^{i_2-i_1}3^{i_3-i_2}\cdots k^{i_{k}-i_{k-1}} (k+1)^{n-i_k},
\end{align*}
and this completes the proof.
\end{proof}

\begin{remark}
\Cref{thm:count for outcomes with forbidden spot one} could be viewed as a corollary of \Cref{thm:count for outcomes} where we let $I' =\{ 1, i_1 +1, i_2 +1,\ldots, i_k+1\}$ be the lucky set. If there are $k+1$ lucky cars, the probability that car $1$ parks in spot $1$ is $\frac{1}{k+1}$, since {the way that outcomes are constructed gives equal likelihood for any of the lucky cars to appear first}.
We can then divide the formula in \Cref{thm:count for outcomes} by $k+1$ to obtain the same result.
\end{remark}

\begin{corollary}\label{cor:first k forbidden spot one}
If $\bsy{u} = (2,3,\ldots,n,  n+1)$ and $I = \{1,2,\ldots,k\}$. Then
\[
|\out_{\bsy u}(I)|=k! (k+1)^{n-k}.
\]
\end{corollary}

\begin{proof}
We write $I = \{i_1,i_2,\ldots, i_k\}$ with $i_1=1,\,i_2=2,\,\ldots,\, i_k=k$. Then by \Cref{thm:count for outcomes with forbidden spot one}, \[|\out_{\bsy u}(I)| =2^{i_2-i_1}3^{i_3-i_2}\cdots k^{i_{k}-i_{k-1}} (k+1)^{n-i_k} = 2^{1}3^{1}\cdots k^{1} (k+1)^{n-k} = k! (k+1)^{n-k}.\qedhere\]

\end{proof}

\subsubsection{Counting outcomes when the only unavailable spot is spot $2$}\label{subsec: spot 2 is unavailable}
In this section, we give a formula for the number of outcomes of $\bsy{u}$-parking functions for $\bsy{u}=(1,3,4,\ldots, n,n+1)$ with lucky set $I=\{i_1,i_2,\ldots,i_k\}$ where $k\leq n$. In this scenario, there are $n+1$ spots on the street with spot $2$ unavailable and the remaining $n$ spots are available, and for any outcome $B\in \out_{\bsy{u}}(I)$, $B_2=\X$.

\begin{definition}
For $j_1 < \dots < j_\ell \in [n]$, we define the set $I_{(j_1,\dots, j_\ell)}$ to be the set constructed from $I \backslash \{j_1, \dots, j_\ell\}$, in which each element $i \in I \backslash \{j_1, \dots, j_\ell\}$ is decreased by $|\{j_1, \dots, j_\ell\} \cap [i]|$, i.e.
\[ I_{(j_1,\dots, j_\ell)} \colonequals \{i -  |\{j_1, \dots, j_\ell\} \cap [i]| : i \in I \backslash \{j_1, \dots, j_\ell\}\}.\]

\end{definition}

\begin{example}
Let $I = \{1,4,5\}$.  Then,
\begin{itemize}
    \item $I_{(1,3)}= \{2, 3\}$ since $I \setminus \{1,3\} = \{4,5\}$, and both entries in $(1,3)$ are less than $4,5$, so the entries $4$ and $5$ are decreased by $2$ each.
    \item $I_{(4)} = \{1, 4\}$ since $I \setminus \{4\} = \{1,5\}$, and $4$ is less than $5$, so $5$ is decreased by $1$.
    \item $I_{(2,3)} = \{1,2,3\}$ since $I \setminus \{2,3\} = \{1,4,5\}$, and $1$, $2$, and $3$ are less than $4$ and $5$, so $4$ and $5$ are decreased by $3$. 
\end{itemize}

\end{example}

\begin{definition}\label{def:notation for number of outcomes}
For a given lucky set of cars $I = \{i_1, i_2 \ldots, i_k\}$ {with $1\leq i_1 <\cdots < i_{k} \leq n$}, we define $c_s(n,I)$ to be the number of outcomes for $\bsy{u} = (1,2,3,\ldots, s-1, s+1, \ldots, n+1)$ where spot $s$ is unavailable and there are $n$ cars.
\end{definition}

\begin{example}
Using the result in \Cref{thm:count for outcomes with forbidden spot one}, we can write $c_1(n,I)$ as follows
\[ c_1(n,I) = 2^{i_2 - i_1}3^{i_3 - i_2} \cdots (k+1)^{n - i_k}.\]
\end{example}

\begin{theorem}
\label{thm:count for outcomes with forbidden spot two}
Suppose $\bsy{u}=(1,3,4,\ldots, n,n+1)$. If $I = \{i_1, \dots, i_k\}$ {with $1\leq i_1 <\cdots < i_{k} \leq n$} and $c_2(n,I)$ denotes the number of outcomes of $\bsy u$-parking functions with lucky set $I$, then 
\[ c_2(n, I) = \sum_{j = 1}^k c_1(n-1, I_{(i_j)}).\]
\end{theorem}

\begin{proof}
    Since spot $1$ is available, we have that the car parking in that spot is always lucky. Suppose that car $i_j$ parks in spot $1$ where $1\leq j \leq k$. 
    
    The rest of the street consists of the unavailable spot 2, followed by spots $3$ through $n+1$, where the cars parking are indexed by the set $[n+1]\setminus \{i_j\}$ with lucky set $I \setminus\{i_j\}$.
    We may view this as a copy of the parking situation described in \Cref{thm:count for outcomes with forbidden spot one}, in which
     we have $n$ spots and the first is unavailable, where the cars in $[n]\setminus \{i_j\}$ are viewed in correspondence with the set $[n-1]$. The corresponding lucky set in this setting can be formed from $I\setminus \{i_j\}$
     as follows: for each car whose index is higher than $i_j$, we decrease its index by $1$. That is, the lucky set is $I_{(i_j)}$. The number of ways for the cars to park in spots $3,\ldots, n+1$ is then given by $c_1 (n-1,I_{(i_j)})$.

     The lucky car $i_j$ can be any car from the set $I = \{i_1,\ldots, i_k\}$. Then in total, the number of possible outcomes is
     \[
     c_2(n, I)=\sum_{j=1}^k c_1 (n-1,I_{(i_j)}),
     \]
     as claimed.
\end{proof}

\subsubsection{Counting outcomes when the only unavailable spot is spot $3$}\label{subsec: spot 3 is unavailable}

We now give a formula for the number of outcomes of $\bsy{u}$-parking functions for $\bsy{u} = (1,2,4, \dots, n+1)$ with lucky set $I = \{i_1, i_2 \ldots, i_k\}$ where $k\leq n$. In this scenario, there are $n+1$ spots on the street with spot $3$ unavailable and the remaining $n$ spots are available, and for any outcome $B\in \out_{\bsy{u}}(I)$, $B_3=\X$. We follow the notation from \Cref{def:notation for number of outcomes}, and let $c_3(n,I)$ denote the number of possible outcomes in this case.

 \begin{theorem}
Suppose $\bsy{u}=(1,2,4,\ldots, n,n+1)$.
    If $I = \{i_1, \dots, i_k\}$ {with $1\leq i_1 <\cdots < i_{k} \leq n$} and $c_3(n,I)$ denotes the number of outcomes of $\bsy u$-parking functions with lucky set $I$, then 
\[ c_3(n,I) = \sum_{j = 1}^k c_2(n-1, I_{(i_j)}) + \sum_{j = 1}^k \left(\sum_{s \in [n] \backslash I, s > i_j} c_1(n-2, I_{(i_j,s)})\right),\]
where $c_2(n,I)$ is as in \Cref{def:notation for number of outcomes}.
\end{theorem}

\begin{proof}
Similarly as before, since spot $1$ is available, the car parking in that spot is always lucky. Suppose this car is car $i_j$. The car that parks in spot 2 may or may not be lucky. If the car in spot 2 is lucky, then for the last $n$ spots $\{2,3,\ldots, n+1\}$ where spot $3$ is unavailable, we have a copy of the parking scenario in \Cref{thm:count for outcomes with forbidden spot two} in which the corresponding lucky set is $I_{(i_j)}$ where we shift the vector $\bsy{u'}=(2,4,\ldots,n,n+1)$ to $\bsy{u}=(1,3,4,\ldots,n-1,n)$. In this case, there are $n-1$ cars and $n-1$ available spots with spot $2$ unavailable. The possible outcomes in this case are then counted by $\sum_{j = 1}^k c_2(n-1, I_{(i_j)})$.

Otherwise suppose the car that parks in spot 2 is unlucky. In this case, the car in spot 2 must be from the set $[n]\setminus I$, and its preference must have been spot 1. In order for it to have parked in spot 2, it must have parked after car $i_j$, so its index must be higher than $i_j$. Let $s$ be the index of the car in spot 2. Then we have $s\in [n]\setminus I$, and $s>i_j$.

Once car $s$ has been chosen, we count the possible options for how the remaining cars occupy the remaining spots (the forbidden spot 3, and then spots 4 through $n+1$). Similarly as in the proof of \Cref{thm:count for outcomes with forbidden spot two}, we can view this as a parking scenario for a street with $n-1$ spots in which the first spot is forbidden, and the cars indexed by $[n-2]$ are viewed in correspondence with $[n] \setminus \{i_j, s\}$ with the same relative order. In this case, the lucky set is $I_{(i_j,s)}$. Then for a fixed choice of $i_j$ and $s$, we have $c_1(n-2,I_{(i_j,s)})$ options for the last $n-1$ spots. Over all possible choice of $i_j$ and $s$, we then have $\sum_{j=1}^k \sum_{s \in [n] \setminus I, s>i_j}c_1(n-2,I_{(i_j,s)})$ options.

Then the total number of possible outcomes is given by
\[ c_3(n,I)=\sum_{j = 1}^k c_2(n-1, I_{(i_j)}) + \sum_{j = 1}^k \sum_{s \in [n] \backslash I, s > i_j} c_1(n-2, I_{(i_j,s)}). \qedhere\]
\end{proof}

\begin{remark}
If we keep moving down the street, only removing one spot at a time, we should obtain a recurrence relation that depends on the previous cases. That is, if the only unavailable spot is $s$ then the recurrence should depend on $c_1(x_1,I_{y_1}), c_2(x_2,I_{y_2}), \ldots, c_{s-1}(x_{s-1},I_{y_{s-1}})$ where $x_1, x_2,\ldots,x_{s-1}$ are some lengths (depending on how many spots are available) and $y_1,y_2,\ldots,y_{s-1}$ are the corresponding sets of cars depending on which cars we are removing from the lucky set $I_{y_i}$. We pose this as an open problem in \Cref{sec:future work}.
\end{remark}

\section{Counting outcomes of vector parking functions with a fixed set of lucky spots}\label{sec:outcomesbyspots}

In this section, given a vector $\bsy u$, we count outcomes corresponding to a given set $L$ of lucky spots, where a parking spot $\ell$ is a \emph{lucky spot} if it is the preference of a lucky car.
By construction, $L$ is a subset of the entries in $\bsy u$. 
Notice that if any car parks in its preferred spot $\ell$, that spot is said to be a lucky spot (even if other cars park in spot $\ell$ that did not prefer $\ell$).
Throughout we let $L$ be a set containing the indices of lucky spots; i.e.\ the spots in which lucky cars park.

\begin{definition}\label{def:delta}
Given a vector $\bsy u$ of length $n$, and $L$ a subset of the entries in $\bsy u$, we associate a composition of $n$, constructed by separating the entries in $\bsy u$ into sub-intervals as follows:
    \begin{enumerate}
    \item Place a bar to the left of the first entry in $\bsy u$. 
    \item For any $i \in \{1,2,\ldots, \max(\bsy u)\}$ such that $i \in \bsy u$ but $i+1 \notin \bsy u$, place a bar to the right of $i$ in $\bsy u$.
    \item For each spot $\ell\in L$ where {$\ell>1$}, place a bar to the left of the first instance of this value in $\bsy u$.
\end{enumerate}
Let $k$ denote the total number of contiguous blocks of unavailable spots plus the number of lucky spots. 
Once the bars have been placed,
for all $1\leq i\leq k$, let $\delta_i$ be the length of the contiguous set of entries between the $i$th and $(i+1)$th bar. 
Now define the \emph{associated composition} $\delta_{\bsy u, L} =  (\delta_1,\ldots, \delta_k)$ to be a weak composition of $n$ which is given via this construction. 

\end{definition}
We illustrate \Cref{def:delta} next.

\begin{example}
Let $n=6$, $\bsy u=(1,1,2,4,4,5)$ and $L=\{1,5\}$. We construct the associate composition by placing the following bars:
 \[ \mid 112 \mid 44  \mid 5 \mid.\]
 
\noindent The resulting associated composition is $\delta_{\bsy u, L}=(3,2,1)$.

If $\bsy u$ is the same and instead we consider $L'=\{1,4\}$, then we would place the following bars:
\[\mid 112 \mid \mid 44  5 \mid.\]
The resulting associated composition is $\delta_{\bsy u, L'}=(3,0,3)$.
\end{example}

\begin{definition}
Viewing the outcome of a $\bsy u$-parking function as a permutation of $[n]$ with $\X$'s inserted to represent unavailable spots between entries, we define the \emph{underlying permutation} to be the sequence of entries with $\X$'s removed. If multiple cars park in the same spot, we represent them as a sequence in increasing order.
\end{definition}

\begin{example}
If the outcome of a $\bsy{u}$-parking function is $\{1\}\X\{2,3,6\}\X\X\{4\}\{7,8\}\X\{5\}$, the underlying permutation (in one-line notation) is $12364785$.
\end{example}
\begin{theorem}
    Let $\bsy u$ be a vector of length $n$, and let $L$ be a subset of the entries in $\bsy u$ indexing the lucky spots. 
    Let $\delta_{\bsy u, L} = (\delta_1,\ldots, \delta_k)$ be the associated composition of $n$.
Then the number of outcomes for $\bsy u$-parking functions with lucky spots indexed by $L$ is  \[\binom{n }{\delta_{\bsy u, L}} = \binom{n }{\delta_1,\ldots,\delta_k} = \frac{n!}{\delta_1!\, \delta_2 ! \,\cdots\, \delta_k!}.\]
\end{theorem}

\begin{proof}

Once the lucky spots are chosen for the lucky cars, the unlucky cars may park in the remaining available spots.
So, an unlucky car may park to the right of a lucky car who parked prior, or to the right of an unavailable spot. Within the available ``unlucky'' spots---i.e., spots not indexed by $L$, the unlucky cars must form increasing subsequences (starting with each lucky car, or unlucky car following an unavailable spot), by the same reasoning as in the proof of Theorem \ref{thm:count for outcomes}. Each car that parks in an unlucky spot must have preferred an earlier spot, and it must have been completely occupied or unavailable before that car checked it. So within each lucky spot, the cars parking in that spot must all have smaller indices than each of the unlucky spots in a contiguous run to the right (even if not all cars in the lucky spot preferred this spot).

By the discussion above, the underlying permutation of an outcome with lucky spots given by a set $L$ will be constructed of increasing runs, which may be broken at the beginning of lucky spots (where the first entry corresponding to the lucky spot is the starting position of a new increasing subsequence) or at unavailable spots (where in the underlying permutation, a new increasing run may start following where an unavailable spot appeared).  By construction, the lengths of these increasing runs are exactly the parts of the associated composition $\delta_{\bsy u, L} =  (\delta_1, \ldots, \delta_k)$. We may alternatively view these objects as permutations with possible descents either at the beginning of a lucky spot, 
or where there are ``gaps'' in which spots are unavailable.

To count such permutations, we refer to a result of MacMahon \cite[Art. 157]{MacMahon}, and obtain that the number of permutations of the desired form is given by the multinomial coefficient \[\binom{n}{\delta_{\bsy u, L}} = \frac{n!}{\delta_1!\, \delta_2 ! \,\cdots\, \delta_k!}.\qedhere\]
\end{proof}
\begin{example}
    Let $\bsy{u} = (1,2,4,5)$.
    Let us count the outcomes in which the lucky spots are given by $L = \{1, 5\}$. 
    To do so, we partition $\bsy{u}$ as follows:
    \[
    \mid 12 \mid 4  \mid 5\mid \;.
    \]
    Then the associated composition $\delta_{\bsy u, L}=(2,1,1)$ and the number of outcomes is $\binom{4 }{ 2,1,1} = \frac{4!}{2!1!1!} = 12$. 
    One can verify that the possible outcomes are
    \[ \left\{
    \begin{tabular}{cccc}
        $\{1\}\{2\} \X \{3\}\{4\}$, &
        $\{1\}\{4\}\X\{2\}\{3\}$, &
        $\{2\}\{3\}\X \{1\}\{4\}$, &
        $\{2\}\{3\} \X\{4\}\{1\}$,\\
        $\{1\}\{2\}\X \{4\}\{3\}$, &
        $\{1\}\{3\}\X\{4\}\{2\}$, &
        $\{2\}\{4\} \X\{1\}\{3\}$, &
        $\{2\}\{4\}\X\{3\}\{1\}$, \\     
        $\{1\}\{3\}\X\{2\}\{4\}$, &
        $\{1\}\{4\}\X \{3\}\{2\}$, &
        $\{3\}\{4\} \X\{1\}\{2\}$, &
        $\{3\}\{4\} \X\{2\}\{1\}$
         \end{tabular}
    \right\}.
    \]
\end{example}

\begin{example}
    Let $\bsy{u} = (2,2,3,5)$.
    Let us count the outcomes in which the lucky spots are given by $L = \{3, 5\}$. 
    To do so, we partition $\bsy{u}$ as follows:
    \[
    \mid 22 \mid 3  \mid\mid 5\mid
    \]
    Then the associated composition is $\delta_{\bsy u, L}=(2,1,0,1)$ and the number of outcomes is again $\binom{4}{2,1,0,1} = \frac{4!}{2!1!0!1!} = 12$. 
    One can verify that the outcomes are
    \[ \left\{
    \begin{tabular}{cccc}
        $\X\{1, 2\} \{3\}\X\{4\}$, &
        $\X\{1, 4\}\{2\}\X\{3\}$, &
        $\X\{2, 3\} \{1\}\X\{4\}$, &
        $\X\{2, 3\} \{4\}\X\{1\}$,\\
        $\X\{1, 2\}\{4\}\X\{3\}$, &
        $\X\{1,3\}\{4\}\X\{2\}$, &
        $\X\{2, 4\} \{1\}\X\{3\}$, &
        $\X\{2, 4\}\{3\}\X\{1\}$, \\     
        $\X\{1, 3\}\{2\}\X\{4\}$, &
        $\X\{1,4\} \{3\}\X\{2\}$, &
        $\X\{3, 4\} \{1\}\X\{2\}$, &
        $\X\{3, 4\} \{2\}\X\{1\}$
         \end{tabular}
    \right\}.
    \]
In this case, the two cars that park in spot 2 are both unlucky but still park in ascending order. The cars in spots 3 and 5 are both lucky and can be chosen as any ordered pair among the four possible cars.
\end{example}

We remark that this construction counts outcomes based on which spots are occupied by lucky cars. However, the actual cars that end up being lucky are not necessarily restricted.

\section{Counting $\bsy{u}$-parking functions with certain sets of lucky cars}\label{sec:uparkingfunctions}
So far, we have looked at counting outcomes for families of certain vector parking functions {with certain lucky sets and lucky spots}. 
In this section, we provide a formula for the number of vector parking functions with a fixed set of lucky cars $I$. We also provide a formula for the number of vector parking functions with a lucky set of fixed size; that is, $|I|=k$. For sake of simplicity in our arguments, we assume that $\X<i$ for any $i\in \mathbb{N}$. 

\subsection{Counting vector parking functions with a fixed set of lucky cars}

Recall that if there exists some $i \notin \bsy{u}$ where $i<\max(\bsy{u})$, then spot $i$ has capacity zero, i.e.\ spot $i$ is unavailable. We begin with a definition.

\begin{definition}
For an ordered set partition $B=B_1B_2\cdots B_M\in \out_{\bsy{u}}(\bsy{a})$ where $M=\max(\bsy{u})$, let $b_i=\max(B_i)$ and if $B_i=\X$ then $b_i=\X$. For each $b\in B_i$, let $s_B(b)$ be the length of the longest subsequence $b_jb_{j+1}\cdots b_{i-1}$ with $b_t<b$  for all $j\leq t\leq i-1$. We denote the length of a sequence $b_1b_{2}\cdots b_{i}$ by $|b_1b_{2}\cdots b_{i}|$.

\end{definition}

\begin{example}\label{ex:outcome2}
Let $\bsy{u}=(1,1,3,3,3)$ and $B=\{2,3\}\X\{1,4,5\}$ where $B_1=\{2,3\}, B_2=\X$ and $B_3=\{1,4,5\}$. Then $b_1=3, b_2=\X$ and $b_3=5$. Then, $s_B(b=2)=0$ and $s_B(b=3)=0$, since $2,3 \in B_1$ and there are no $b_i$'s with index $i$ less than $1$. Now, for each $b\in B_3$, we have the following. For $b=1$, since $\X=b_2<1$ and $3=b_1\nless 1$, we have $s_B(b=1)=|b_2|=1$. For $b=4$, since $\X=b_2<4$ and $3=b_1<4$, it follows that $s_B(b=4)=|b_1b_2|=2$. Similarly, for $b=5$, since $\X=b_2<5$ and $3=b_1<5$, we have $s_B(b=5)=|b_1b_2|=2$. It follows that $s_B(b=1)=1, s_B(b=2)=0, s_B(b=3)=0,  s_B(b=4)=2$ and $s_B(b=5)=2$.
\end{example}

For each $b\in B_i$, the definition $s_B(b)$ counts the number of distinct spots where cars arriving before car $b$ parked contiguously and immediately to the left of car $b$ (if any) and the number of spots that have capacity $0$ contiguously and immediately to the left of car $b$ (if any). We stress that $s_B(b)$ is the number of spots and not the number of cars. We show (in \Cref{lem:possible prefs given an outcome}) that in the case that car $b$ is an unlucky car, the number $s_B(b)$ is exactly the number of potential distinct spots car $b$ could prefer and which car $b$ would find either occupied by cars that arrived and parked before car $b$ or spots that have capacity $0$. We make this precise next.

\begin{definition}
Fix a lucky set $I$ of $\PF(\bsy{u})$ and an outcome $B=B_1B_2\cdots B_M\in \out_{\bsy{u}}(I)$ where $M=\max(\bsy{u})$. For each $b\in B_i$, let $\Pref_I(b,\bsy{u})$ denote the set of preferences of car $b$ in a $\bsy{u}$-parking function $\bsy{a}$ of length $n$ satisfying $\Lucky_{\bsy{u}}(\bsy{a})=I$ and $\out_{\bsy{u}}(\bsy{a})=B$.
\end{definition}

\begin{example}[Continuing \Cref{ex:outcome2}]\label{ex:l function}
If $B=\{2,3\}\X\{1,4,5\}$ and if car $1$ is unlucky, then it could only prefer the unavailable spot to its left and it cannot prefer spot $1$ as it would find it empty and park there. 
This agrees with $s_B(b=1)=1$. If car $4$ is unlucky, then it could only prefer spot $1$, which is occupied by cars $2$ and $3$, or the unavailable spot $2$. This is true because cars $2$ and $3$ are parked to the left of car $4$ and arrived prior to car $4$, and spot $2$ is unavailable, so car $4$ could have preferred this spot and be unlucky. This agrees with $s_B(b=4)=2$.
\end{example}

The count for the number of preferences of unlucky cars detailed in \Cref{ex:l function} holds in general.

\begin{lemma}\label{lem:possible prefs given an outcome}
Fix a lucky set $I$ of $\PF(\bsy{u})$ and fix $B=B_1B_2 \cdots B_M\in\out_{\bsy{u}}(I)$, where $M=\max(\bsy{u})$. Then for each $b\in B_i$,
\[|\Pref_I(b,\bsy{u})| 
    = \begin{cases}
      1&\text{ if } b\in I\\
      s_B(b) & \text{ if } b \notin I,
     \end{cases}
    \]
    and the number of possible $\bsy{u}$-parking functions $\bsy{a}$ with outcome $B$ and lucky set $I$ is equal to
\[\prod_{b\in [n]} |\Pref_{I}(b,\bsy{u})|=\prod_{b\notin I}s_B(b).\]

\end{lemma}

\begin{proof}
We count possible preferred parking spots for each car $b \in [n]$.
If $b$ is a lucky car, its preference must be the spot in which it parked. So for  $b \in I$, $|\Pref_{I}(b,\bsy{u})| = 1$.

Otherwise, suppose car $b$ is unlucky. In this case, its preference must have been to the left of the spot in which it parked. 
We look to the left at all spots which it could have preferred. It could not have preferred a spot that is occupied by a car of higher index, since if that was the case, then it would have parked in that spot first. Similarly, it could not have preferred any spot to the left of a car of higher index. 

So we look only at consecutive spots to the left, between the spot containing car $b$ and the spot containing the nearest car of higher index. If no car of higher index appears to the left of the spot where car $b$ parked, we consider all spots to the left.

Among spots to the left that are either unavailable or occupied completely by cars of lower index, the unlucky car $b$ could have preferred any of these to end up parking where it did. So each of these spots are possible preferences for car $b$, and they are counted by $s_B(b)$. So we have, for $b \notin I$, $|\Pref_{I}(b,\bsy{u})|= s_B(b)$. 

Then we may count the total number of parking functions corresponding to this outcome as the product of possible preferences for each of the $n$ cars; that is, \[\prod_{b\in [n]} |\Pref_{I}(b,\bsy{u})|.\]

Since $|\Pref_{I}(b,\bsy{u})| = 1$ for all $b \in I$, these terms do not change the product so we may restrict attention to all $b \notin I$. The formula reduces to \[\prod_{b\notin I}s_B(b). \qedhere\]
\end{proof}

\begin{example}[Continuing \Cref{ex:outcome2}]\label{ex:pref}
 Suppose that $I=\{2,3,4\}$ and $B=\{2,3\}\X\{1,4,5\}$. 
 By \Cref{lem:possible prefs given an outcome}, the number of $\bsy{u}$-parking functions $\bsy{a}$ satisfying $\out_{\bsy{u}}(\bsy{a})=B$ and $\Lucky_{\bsy{u}}(\bsy{a})=I$ is given by
\[
     \prod_{b\in [n]} |\Pref_{I}(b,\bsy{u})|=\prod_{b\notin I}s_B(b)= s_B(b=1)\cdot s_B(b=5) =2.
\]
Indeed, if $\bsy{u}=(1,1,3,3,3)$ then the only $\bsy{u}$-parking functions with lucky set $I=\{2,3,4\}$ and outcome $B=\{2,3\}\X\{1,4,5\}$ are
\[
\bsy{a}=(2,1,1,3,\fbox{1}) \mbox{ and } \bsy{a}'=(2,1,1,3,\fbox{2}), 
\]
where we box the difference between these $\bsy{u}$-parking functions, which denote the possible preferences of car $5$ being spots $1$ and $2$. 
Since cars $2$, $3$, and $4$ are lucky, their preferences are fixed to be the spots where they parked.
Since car $1$ parks in spot $3$ and is unlucky, and since cars $2$ and $3$ park in spot $1$, we have that car $1$ may only prefer spot $2$.
Since car $5$ is unlucky, and spot $1$ is fully occupied when car $5$ goes to park and spot $2$ is unavailable, car $5$ may prefer either spot $1$ or spot $2$.
\end{example}

We now give a formula for the number of $\bsy{u}$-parking functions with a fixed set of lucky cars.

\begin{theorem}
Fix a lucky set $I$ of $\PF(\bsy{u})$. Let $\out_{\bsy{u}}(I)$ denote the set of outcomes $B=B_1B_2\ldots B_M$ where $M=\max(\bsy{u})$ and lucky set $I$ where $B_i$ denotes the set of cars that parked in spot $i$ and if $B_i=\X$ then spot $i$ is unavailable (that is, no cars can park there). 
Then
\[|\LPF_{\bsy{u}}(I)|=\sum_{B\in \out_{\bsy{u}}(I)} \left( \prod_{b\notin I}s_B(b)\right).\]
\end{theorem}

\begin{proof}
\Cref{lem:possible prefs given an outcome} established that for a specific outcome and lucky set, the number of $\bsy{u}$-parking functions yielding that outcome and lucky set is given by $\prod_{b\notin I}s_B(b)$. Since each $\bsy{u}$-parking function corresponds to exactly one outcome, we may compute the total number of $\bsy{u}$-parking functions with a given lucky set $I$ by summing over all possible outcomes with lucky set $I$, and counting the $\bsy{u}$-parking functions corresponding to each. The result follows.      
\end{proof}

\subsection{Counting vector parking functions when $k$ cars are lucky}
In this section, we give a formula for the number of $\bsy{u}$-parking functions with a lucky set of size $k$. {For ease of our arguments, our convention is that $0^0=1$}. 
In this section, we view an outcome of a vector parking function as usual, but without the unavailable spots $\X$ and we reindex the nonempty parts in the outcome. This will make the outcome an ordered set partition of $[n]$. See the following for an example.

\begin{example}\label{ex: dropping the unavailable spots}
If $\bsy{u}=(3,3,5,5)$ and $\out_{\bsy{u}}(\bsy{a})=\X\X\{2,3\}\X\{1,4\}$ for some $\bsy{a}$, then the outcome becomes $B=\{2,3\}\{1,4\}$ where $B_1=\{2,3\}$ and $B_2=\{1,4\}$.
\end{example}

\noindent We first set some notation as follows. Let 
\[\bsy{u}=(\underbrace{u_1,u_1,\ldots, u_1}_{m_1}, \underbrace{u_2,u_2,\ldots, u_2}_{m_2}, \ldots, \underbrace{u_r,u_r,\ldots ,u_r}_{m_r}),\]
where for each $u_j\in \bsy{u}$, $u_j$ has multiplicity $m_j$ for all $j\in [r]$.

Suppose there are $k$ lucky cars. Let $B=(B_1,B_2,\ldots,B_r)$ be an ordered set partition of the interval of integers $[\sum_{j=1}^r m_j]=\{1,2,\ldots, m_1+m_2+\cdots+m_r\}$ where $|B_j|=m_j$ (each $B_j$ contains the cars that will park in spot $i_j$). Let $\kappa = (k_1,\ldots, k_r)$ be a partition of $k$ such that $k_j$ denotes the number of lucky cars in set $B_j$. Moreover, let $L_j$ be the set that contains the lucky cars in $B_j$ where $|L_j|=k_j$.

\begin{definition}
For $0<t<j$, let $\ell_{j,t} $ be the number of unlucky cars in set $B_j$ such that when they park, spots {$u_t$} through {$u_{j-1}$} are completely occupied. We define $\ell_{j,j}=|B_j\setminus L_j|=m_j-k_j$ (total number of unlucky cars in $B_j$) and $\ell_{j,0}=0$. 
\end{definition}
\begin{lemma}
For $0<t<j$, the count $\ell_{j,t} $ is given by
    \[
\ell_{j,t} \coloneq \left|\left\{ b\in B_j \; : \; b>\max\left(\bigcup_{d=t}^{j-1}B_d\right) \right\}  \cap \left(B_j\setminus L_j\right)\right| .
\]

\end{lemma}

\begin{proof}

Let $b$ be a car included in the set counted by $\ell_{j,t}$.
When this car parks, every spot from {$u_t$} to {$u_{j-1}$} is filled. So every car included in those spots in the outcome has already parked; therefore they are all smaller in index than car $b$. That is, $b>\max\left(\bigcup_{d=t}^{j-1}B_d\right)$. Also, in order for car $b$ to be included in this count, it must be in the set $B_j$ and be unlucky, so it is in $B_j \setminus L_j$. This is accounted for with the intersection and the result follows.
\end{proof}

\begin{theorem}\label{thm:countforupfs}
Let $\bsy{u}=(\underbrace{u_1,u_1,\ldots, u_1}_{m_1}, \underbrace{u_2,u_2,\ldots, u_2}_{m_2}, \ldots, \underbrace{u_r,u_r,\ldots ,u_r}_{m_r})$ where for each $u_j\in \bsy{u}$, $u_j$ has multiplicity $m_j$ for all $j\in [r]$. Then the number of $\bsy{u}$-parking functions with $k$ lucky cars is
\[ 
\left|\bigcup_{|I|=k} \LPF_{\bsy{u}}(I)\right| = 
\sum_{{\substack{\kappa = (k_1,k_2,\ldots, k_r), \\
\sum_{i=1}^{r}k_i=k,\\
0\leq k_i\leq m_i}}}
\,\sum_{{B = (B_1,B_2,\ldots, B_r) \in \out_{\bsy u}}}\binom{m_1}{k_1}(u_1-1)^{m_1-k_1} S(B ; \kappa),
\]
where
\[S(B;\kappa) \coloneq \sum_{{\substack{(L_2,L_3,\ldots,L_r), \\ L_d\subseteq B_d, |L_d|=k_d}}}\prod_{j=2}^r \prod_{t=1}^{j} (u_{j}-u_{t-1}-1)^{\ell_{j,t}-\ell_{j,t-1}}.\]

\end{theorem}

\begin{proof}
Recall that in this subsection, $B_i$ is no longer the set of cars that parked in spot $i$, but given the reindexing it is now is the $i$th nonempty set that appears in the outcome. 
By assumption, we have $r$ spots with capacities given by $m_1,\ldots, m_r$. Fix an outcome $B =(B_1,B_2,\ldots, B_r)\in\out_{\bsy{u}}$, and suppose that there are $k$ lucky cars. Let $\kappa = (k_1,k_2,\ldots, k_r)$ be a weak composition of $k$, such that the $i$th spot contains $k_i$ lucky cars. Then we have that $0 \leq k_i \leq m_i$ for all $i$.
We count the number of parking functions corresponding to this outcome.

The first available spot is $u_1$.
We have that $B_1$ is the set of cars which parked in spot $u_1$, and $k_1$ of them are lucky. 
First we choose which cars are lucky; there are $ \binom{m_1}{k_1}$ options. For these cars, their preference is fixed to be $u_1$. 

Then there are $m_1-k_1$ unlucky cars remaining in $B_1$. We count the possible preferences for the unlucky cars in this set. 
For each such car, their preference could have been any spot earlier than $u_1$, so they each have $(u_1-1)$ options for their preference. Hence, the number of ways to count the preferences of cars in the set $B_1$ is given by 
$\binom{m_1}{k_1}(u_1-1)^{m_1-k_1}$.

We now consider the possible preferences for the cars in spots $B_2,\ldots, B_r$. Let $S(B,\kappa)$ denote the number of possible preferences among the cars parking in the remaining $r-1$ available spots.

Fix $u_j$ to be one of the spots after $u_1$, and let $L_j$ be the set of lucky cars in spot $u_j$ (so $L_j$ is a subset of $B_j$ consisting of $k_j$ cars). 
For the cars in $L_j$, their preference is fixed to be $u_j$ (the spot where they parked). 
We consider sets of  possible preferences for the remaining $m_j-k_j$ unlucky cars.

Fix a spot $u_t\leq u_{j-1}$ and define $u_0 =0$. 
We count the number of unlucky cars $c$ in $B_j$ whose set of possible preferences contains any spot between $u_{t-1}$ and $u_j$, not including $u_{t-1}$ (or $u_j$, since $c$ is an unlucky car). When such a car $c$ went to park, its preference was after spot $u_{t-1}$, so it did not check this spot, and all spots from $u_t$ to $u_{j-1}$ were full. However, spot $u_{t-1}$ would not have been full since if it was, then car $c$ could have preferred spot $u_{t-1}$ as well. 

The number of such cars is given by $\ell_{j,t}-\ell_{j,t-1}$. These cars could have preferred any spot in the set $\{u_{t-1}+1, ,u_{t-1}+2,\ldots, u_{j}-1\}$, and in particular not any spot $u_{t-1}$ or earlier, to end up parking in spot $u_j$.

For each such car, there are {$u_{j}-1 - (u_{t-1} +1 ) +1 = u_{j} - u_{t-1} -1 $} options for its preference. 
Then the number of possible preferences among all cars which could have preferred any spot $u_t \leq u_{j-1}$ is given by $\prod_{t=1}^{j-1} (u_{j}-u_{t-1}-1)^{\ell_{j,t}-\ell_{j,{t-1}}}$.

The unlucky cars that remain to be accounted for are those that preferred a spot between $u_{j-1}$ and $u_j$, but were unable to park there since it was unavailable.
There are {$u_j-u_{j-1}-1$} unavailable spots between $u_{j-1}$ and $u_j$. To avoid double-counting cars that were accounted for in the previous product, we only consider cars such that when they parked, spot $u_{j-1}$ was available; in other words spot $u_{j-1}$ was not full yet. This guarantees that their preference could not have been $u_{j-1}$ or earlier (or else they would have parked there instead).

The number of unlucky cars such that spot {$u_{j-1}$} was {available} when they parked is given by $m_j-k_j-\ell_{j,j-1}$. 
It follows that $(u_j-u_{j-1}-1)^{m_j-k_j-\ell_{j,j-1}}$ is the number of possible preferences among unlucky cars preferring a spot after $u_{j-1}$ and parking in spot $u_j$.
In total we have, for a fixed spot $u_j$, $$ \left(\prod_{t=1}^{j-1} (u_{j}-u_{t-1}-1)^{\ell_{j,t}-\ell_{j,t-1}}\right)\left(u_j-u_{j-1}-1\right)^{m_j-k_j-\ell_{j,j-1}}$$ possible preferences for the cars in the set $B_j$.
With the convention that $\ell_{j,j} = m_j - k_j$, this product simplifies to
\[ \prod_{t=1}^{j} (u_{j}-u_{t-1}-1)^{\ell_{j,t}-\ell_{j,t-1}}.\] 
Taking the product over all $j=2,\ldots, r$, we have for
a given outcome $B = (B_1,B_2,\ldots, B_r)$,  with lucky cars in each of the spots given by $(L_2,\ldots, L_r)$, a total of
\[ \prod_{j=2}^r \prod_{t=1}^{j} (u_{j}-u_{t-1}-1)^{\ell_{j,t}-\ell_{j,t-1}}\] possible preference lists for the cars in spots $2$ through $r$.
Summing over all possible distributions of lucky cars $(L_2,\ldots, L_r)$, we have 
\[S(B;\kappa) = \sum_{\substack{(L_2,L_3,\ldots,L_r), \\L_d\subseteq B_d, |L_d|=k_d}}\prod_{j=2}^r \prod_{t=1}^{j} (u_{j}-u_{t-1}-1)^{\ell_{j,t}-\ell_{j,t-1}}\]
total options for the preferences of cars parking in the last $r-1$ spots. Then for a fixed outcome $B = (B_1,B_2,\ldots, B_r) \in \out(\bsy u)$, the total number of corresponding parking functions is
\[
\binom{m_1}{k_1}(u_1-1)^{m_1-k_1} S(B ; \kappa).
\]
Since each parking function corresponds to one outcome, we may enumerate all parking functions with $k$ lucky cars by summing over all possible outcomes with $k$ lucky cars. Then the number of parking functions of the desired form is 

\[
\sum_{\substack{\kappa = (k_1,\ldots, k_r),\\0\leq k_i\leq m_i,\\\sum_{i=1}^{r}k_i=k}}
\,\sum_{{B = (B_1,B_2,\ldots, B_r) \in \out(\bsy u)}}\binom{m_1}{k_1}(u_1-1)^{m_1-k_1} S(B ; \kappa). \qedhere
\]
\end{proof}

With restrictions on the form of $\bsy{u}$, the formula in \Cref{thm:countforupfs} can be simplified further.

\begin{corollary}
Let
$\bsy{u}=(\underbrace{i,i,\ldots, i,}_{n-1 \text{ times }} j)$. Then the number of $\bsy{u}$-parking functions with $k$ lucky cars is
\[ 
\left|\bigcup_{|I|=k} \LPF_{\bsy{u}}(I)\right| = \binom{n-1}{k}(i-1)^{n-1-k} \left((n-1)  (j-i-1)+ (j-1) \right)+n\binom{n-1}{k-1}(i-1)^{n-k}.
\]
\end{corollary}

\begin{proof}
We apply the formula in \Cref{thm:countforupfs}, with $m_1 = n-1$, $m_2 =1$, $u_1=i$, and $u_2=j$.
Suppose $a$ is the car that ends up parking in spot $j$. Then $B_2=\{a\}$. Either $a$ is a lucky car or an unlucky car. Hence, $k_2 = 0$ or $k_2 =1$. If $k_2=0$ then all of the lucky cars park in spot $i$, and if $k_2=1$ then one lucky car parks in spot $j$ and the remaining $k-1$ lucky cars park in spot $i$.
The summation is then

\begin{equation}
  \label{eq:corollaryupfs}
\underbrace{\sum_{B_2 = \{a\}}\binom{n-1}{k_1}(i-1)^{n-1-k_1} S(B ; \kappa)}_{k_1=k,k_2=0} + \underbrace{\sum_{B_2 = \{a\}}\binom{n-1}{k_1}(i-1)^{n-1-k_1} S(B ; \kappa)}_{k_1=k-1,k_2=1},
\end{equation}
and for a composition $\kappa = (k_1, k_2)$, we have \begin{align*}
S(B;\kappa) &= \sum_{{\substack{(L_2,L_3,\ldots,L_r),\\L_d\subseteq B_d, |L_d|=k_d}}}\prod_{j=2}^r \prod_{t=1}^{j} (u_{j}-u_{t-1}-1)^{\ell_{j,t}-\ell_{j,t-1}} \\
&=\sum_{{\substack{(L_2),\\L_2\subseteq B_2, |L_2|=k_2}}}\prod_{j=2}^2 \prod_{t=1}^{j} (u_{j}-u_{t-1}-1)^{\ell_{j,t}-\ell_{j,t-1}}\\
&=\sum_{L_2\subseteq B_2, |L_2|=k_2} \prod_{t=1}^{2} (u_{2}-u_{t-1}-1)^{\ell_{2,t}-\ell_{2,t-1}}\\
&=\sum_{L_2\subseteq B_2, |L_2|=k_2} (u_{2}-u_{0}-1)^{\ell_{2,1}-\ell_{2,0}} (u_{2}-u_{1}-1)^{\ell_{2,2}-\ell_{2,1}} \\
&=\sum_{L_2\subseteq B_2, |L_2|=k_2} (j-1)^{\ell_{2,1}} (j-i-1)^{\ell_{2,2}-\ell_{2,1}}.
\end{align*}

Then the summation in \eqref{eq:corollaryupfs} becomes
\begin{align}
\notag
&\underbrace{\sum_{B_2 = \{a\}}\binom{n-1}{k}(i-1)^{n-1-k} \sum_{L_2 = \emptyset}  (j-1)^{\ell_{2,1}} (j-i-1)^{\ell_{2,2}-\ell_{2,1}}}_{k_1=k,k_2=0}\\
&+ \underbrace{\sum_{B_2 = \{a\}}\binom{n-1}{k-1}(i-1)^{n-1-k+1} \sum_{L_2= B_2}  (j-1)^{\ell_{2,1}} (j-i-1)^{\ell_{2,2}-\ell_{2,1}}}_{k_1=k-1,k_2=1}.\label{eq:will simplify}
\end{align}
The sum $\sum_{L_2= B_2} (j-1)^{\ell_{2,1}} (j-i-1)^{\ell_{2,2}-\ell_{2,1}}$ is equal to $1$ since there is only one instance when $L_2 = B_2$ and in this case, $\ell_{j,t}=0$ always since there are no unlucky cars in the set to count. Hence, the two underbraced terms in \eqref{eq:will simplify} become
\begin{align}
&\underbrace{\sum_{B_2 = \{a\}}\binom{n-1}{k}(i-1)^{n-1-k} \sum_{L_2 = \emptyset}  (j-1)^{\ell_{2,1}} (j-i-1)^{\ell_{2,2}-\ell_{2,1}}}_{k_1=k,k_2=0}+ \underbrace{\sum_{B_2 = \{a\}}\binom{n-1}{k-1}(i-1)^{n-k} }_{k_1=k-1,k_2=1} \notag\\
&= \underbrace{\sum_{B_2 = \{a\}}\binom{n-1}{k}(i-1)^{n-1-k} \sum_{L_2 = \emptyset} (j-1)^{\ell_{2,1}} (j-i-1)^{\ell_{2,2}-\ell_{2,1}}}_{k_1=k,k_2=0}+ \underbrace{n\binom{n-1}{k-1}(i-1)^{n-k} }_{k_1=k-1,k_2=1} .\label{eq:left sum}
\end{align}
We break up the left sum in \eqref{eq:left sum}
and arrive at
\begin{align}
    &\sum_{B_2 = \{a\}}\binom{n-1}{k}(i-1)^{n-1-k} \sum_{L_2 = \emptyset}  (j-1)^{\ell_{2,1}} (j-i-1)^{\ell_{2,2}-\ell_{2,1}}\notag\\
&=\binom{n-1}{k}(i-1)^{n-1-k}\left(\sum_{\substack{B_2 = \{a\},\\ 1\leq a\leq n-1}}
 \sum_{L_2 = \emptyset}  (j-1)^{\ell_{2,1}} (j-i-1)^{\ell_{2,2}-\ell_{2,1}}+\sum_{B_2 = \{n\}}\sum_{L_2 = \emptyset}  (j-1)^{\ell_{2,1}} (j-i-1)^{\ell_{2,2}-\ell_{2,1}}\right).\label{eq:left sum again}
\end{align} 
When $L_2 = \emptyset$ we have $\ell_{2,2}=1$. If $B_2 = \{a\}$ for $1\leq a \leq n$ then $\ell_{2,1}=0$, since spot $i$ will not be completely occupied when car $a$ parks. 
Otherwise, if  $B_2 = \{n\}$ then $\ell_{2,1}=1$ since the car in $B_2$ is the last car to park.
The sum in \eqref{eq:left sum again} simplifies to
\begin{align*}
&\binom{n-1}{k}(i-1)^{n-1-k}\left(\sum_{\substack{B_2 = \{a\},\\ 1\leq a\leq n-1}}
 \sum_{L_2 = \emptyset}  (j-1)^{0} (j-i-1)^{1-0}+\sum_{B_2 = \{n\}}\sum_{L_2 = \emptyset}  (j-1)^{1} (j-i-1)^{1-1} \right)\\
 &=\binom{n-1}{k}(i-1)^{n-1-k}\left(\sum_{\substack{B_2 = \{a\},\\ 1\leq a\leq n-1}}
 \sum_{L_2 = \emptyset}  (j-i-1)+\sum_{B_2 = \{n\}}\sum_{L_2 = \emptyset}  (j-1) \right)\\
 &=\binom{n-1}{k}(i-1)^{n-1-k} \left((n-1)  (j-i-1)+ (j-1) \right).
\end{align*} 
In total we obtain
\[\binom{n-1}{k}(i-1)^{n-1-k} \left((n-1)  (j-i-1)+ (j-1) \right)+n\binom{n-1}{k-1}(i-1)^{n-k}
\]
$\bsy{u}$-parking functions with $k$ lucky cars.

\end{proof}

\begin{remark}
    Instead of applying \Cref{thm:countforupfs}, the previous result can be obtained from a purely combinatorial argument found by considering possible cases. The resulting count can be written as 
    \[\binom{n-1}{k}(i-1)^{n-1-k} (n-1)  (j-i-1)+ \binom{n-1}{k}(i-1)^{n-1-k}(j-1) +n\binom{n-1}{k-1}(i-1)^{n-k}.
\]
The first term of the above sum counts the possible $\bsy u$-parking functions when all of the $k$ lucky cars park in spot $i$ and the unlucky car in spot $j$ is not the $n$th car, and the second term counts the possible $\bsy u$-parking functions when all of the $k$ lucky cars park in spot $i$ and the unlucky car in spot $j$ is the $n$th car. Lastly, the third term counts the possible $\bsy u$-parking functions when $k-1$ lucky cars park in spot $i$ and one lucky car parks in spot $j$.

\end{remark}

\begin{remark}
In the case of $\bsy u = (1,2,\ldots, n)$, \Cref{thm:countforupfs} gives the number of classical parking functions with $k$ lucky cars, providing an alternative formula to compute the coefficient on $q^k$ in (\ref{eq:gesselseo}).
\end{remark}

\section{Further directions}\label{sec:future work}
We conclude with some directions for further study.
\begin{itemize}
    \item 
In \Cref{subsec: k cars lucky no rep}, we considered the case where the vector $\bsy u$ has no repetition and enumerated the possible outcomes that have $k$ lucky cars. 
It remains an open problem to count the outcomes with $k$ lucky cars for a general $\bsy u$. A place to start would be to consider cases where the vector $\bsy u$ has some repetitions, and count outcomes with a fixed lucky set or a fixed number of lucky cars.

\item In Section \ref{sec:one unavailable spot}, we considered streets with exactly one unavailable spot and no repetition, and we counted outcomes for a fixed set of lucky cars. 
We give an explicit formula for the case when the first spot is unavailable, and then recursive formulas for the cases when spot 2 or spot 3 is unavailable. Based on those results we remark that for a general unavailable spot $i>3$, it should be possible to obtain a recursive formula for the number of outcomes for a fixed set of lucky cars. It remains an open problem to find a closed formula for this count.

\item In \Cref{sec:uparkingfunctions},  we provide a formula for the number of vector parking functions with a fixed set of lucky cars. It would be of interest to determine closed formulas 
for the results in this section as well as for the specializations presented. 

\end{itemize}

One generalization of parking functions is the set of $(m,n)$-parking functions, where there are $m$ cars and $n$ parking spots, with $m\leq n$, and cars follow the standard parking rule to park on the street. 
It would be of interest to introduce a generalization of vector parking functions in which there is more capacity for cars than cars parking on the street. 
With such a generalization defined, one direction for further study is to characterize and enumerate the outcomes of these parking functions with a fixed set of lucky cars.

Another generalization of parking functions considers cars with various lengths, which are called parking sequences and parking assortments \cites{EhrenborgHapp,assortments}. The difference between sequences and assortments is the parking rule. In parking sequences, cars attempt to park in a contiguous section of the street, and if they fit they park; otherwise they cause a collision and exit the street. In parking assortments, the parking rule allows for the cars to seek forward in the street to find a contiguous segment of parking spots on the street in which the car can park. It would be of interest to define a version of vector parking sequences and vector parking assortments, where spots have a certain capacity and cars have specified lengths and parking preferences. With these generalizations, one could study the lucky statistic on these parking functions, as well as the possible outcomes that arise when a set of lucky cars is specified.

\section*{Acknowledgments}

L.~Martinez gratefully acknowledges William \& Mary for hosting her research visit, and acknowledges support by the National Science Foundation Graduate Research Fellowship Program under Grant No.~2233066. This work was also supported by a grant from the Simons Foundation (Travel Support for Mathematicians, P.~E.~Harris).

\bibliographystyle{plainurl}
\bibliography{references.bib}

\end{document}